\documentclass[reqno]{amsart}

\usepackage{amssymb, longtable,comment}
\usepackage[mathcal]{euscript}
\usepackage{color,cite}

\usepackage{hyperref}                   
\hypersetup{colorlinks,
    linkcolor=blue,%
    citecolor=blue}

\usepackage{amssymb, longtable,mathrsfs}
\usepackage{amsmath,amsfonts,amssymb,amsthm,amscd,latexsym}
\usepackage{color}
\usepackage{cite}
\usepackage{graphicx}

\usepackage{hyperref}					
\hypersetup{colorlinks,
	linkcolor=blue,%
	citecolor=blue}	

\newtheorem{theorem}{Theorem}[section]
\newtheorem{lemma}[theorem]{Lemma}

\newtheorem{example}[theorem]{Example}
\newtheorem{prop}[theorem]{Proposition}
\newtheorem{corollary}[theorem]{Corollary}

\newtheorem{remark}[theorem]{Remark}

\newcommand{\cM}{{\mathcal M}}
\newcommand{\cN}{{\mathcal N}}
\newcommand{\cA}{{\mathcal A}}
\newcommand{\cB}{{\mathcal B}}

\newcommand{\cR}{{\mathcal R}}

\newcommand{\cL}{{\mathcal L}}

\newcommand{\bC}{{\mathbb C}}

\begin{document}

\author[J. Huang]{Jinghao Huang}
\address{Institute for Advances Study in Mathematics, Harbin Institute of Technologies,
Harbin, 150001, China}
\email{jinghao.huang@hit.edu.cn}

\author[K. Kudaybergenov]{Karimbergen Kudaybergenov}
\address{Institute for Advances Study in Mathematics, Harbin Institute of Technologies,
Harbin, 150001, China and
Suzhou Research Institute of Harbin Institute of Technology, Suzhou, 215104, China}
\email{kudaybergenovkk@gmail.com}

\author[F. Sukochev]{Fedor Sukochev}
\address{School of Mathematics and Statistics, University of New South Wales,
Kensington, 2052, Australia}
\email{f.sukochev@unsw.edu.au}

\title[Rank metric isometries and determinant-preserving mappings]{Rank metric isometries and determinant-preserving mappings  on II$_1$-factors}

\begin{abstract}We fully describe the general form of  a linear (or conjugate-linear) rank metric isometry on the Murray--von Neumann  algebra  associated with a II$_1$-factor.  
As an application, we establish Frobenius' theorem in the setting of II$_1$-factors, by showing that 
every   determinant-preserving  linear bijection between two II$_1$-factors is necessarily an isomorphism or an anti-isomorphism.
This confirms the Harris--Kadison conjecture (1996). 
\end{abstract}

\subjclass[2010]{47L60, 47C15, 16E50}
\keywords{Murray--von Neumann algebra; determinant; rank metric; isometry}

\maketitle

\bigskip

\section{Introduction}
\subsection{Harris--Kadison Conjecture concerning determinant-preserving bijections}
The starting point is the following classical result due to Frobenius\cite{Fro} in 1897:
\begin{quote}
    a mapping that preserves the determinant is the composition of an automorphism or an anti-automorphism of $\mathbb{M}_n$, $n\times n$ complex-valued matrices, with a left multiplication by a matrix of determinant $1$.
\end{quote}
The notion of the Fuglede-Kadison determinant, extending the classical determinant on $\mathbb{M}_n$ to finite von Neumann algebras,  was introduced in \cite{FK52}. 
Recall that every II$_1$-factor $\cM$ has a unique faithful normal tracial state $\tau.$
The Fuglede--Kadison determinant\cite{FK52} ${\rm det}: \cM \to  [0,\infty)$ is given
by
\begin{align*}
{\rm det}(x)  =\lim\limits_{\varepsilon\downarrow 0}{\rm exp}\left(\tau({\rm log}(|x|+\varepsilon\mathbf{1}))\right),\, x\in \cM.
\end{align*}
In \cite{HS07}, Haagerup and Schultz defined the Fuglede--Kadison determinant on the algebra $\mathcal{L}_{\log } (\cM,\tau_\cM)$ for any semifinite von Neumann algebra $\cM$ equipped with a faithful normal semifinite trace $\tau_\cM$  (see \cite{DSZ15,DSZ16,DNZ17} and  Section \ref{sec:proof} below). 
Nowadays, the Fuglede--Kadison determinant serves as a powerful tool across various fields, such as the topological invariant ``$L_2$-torsion''\cite{Luck,L12,LT14}, invariant subspaces for operators in II$_1$-factors\cite{HS09}, entropy of certain algebraic actions\cite{Bowen,H16,KL11}, Mahler measures and Ljapunov exponents\cite{Deninger} and 
Schur-type upper triangular forms of tuples of commuting operators\cite{CDSZ20} etc.

In their study of affine mappings of invertible operators, 
  Harris and Kadison 
  suggested 
  to establish Frobenius' theorem for II$_1$-factors
\cite[Conjecture 3.2]{HK96}:
\begin{quote}
  \it   If $\cM$ and $\cN$ are factors of type II$_1,$ a linear isomorphism (bijection) $\Phi$ of
$\cM$ onto $\cN$ that maps the unit of $\cM$ to that of $\cN$ preserves determinants if and
only if it is an (algebraic) isomorphism or an anti-isomorphism.
\end{quote}

The main aim of the present paper is to establish  Frobenius's theorem in the setting of II$_1$-factors, which
answers 
the Harris--Kadison Conjecture above in the affirmative.

\begin{theorem}\label{HKcon}
Let $\cM$ and $\cN$ be factors of type II$_1$  and let  $\Phi$  be a unital linear bijection of
$\cM$ onto $\cN$.  Then $\Phi$ preserves determinants if and
only if it is an  (linear) isomorphism or an anti-isomorphism. 
More precisely,
\begin{align}\label{precise form}
\Phi(x)=aJ(x)a^{-1},\, x\in \cM,
\end{align}
where $a$ is an invertible element in $\cN$ and $J$ is a  (linear)   Jordan $*$-isomorphism from $\cM$ onto
$\cN.$
\end{theorem}

It is shown  by Aupetit that any linear spectrum-preserving bijective map between two von Neumann algebras is a Jordan isomorphism (see \cite[Theorem 1.3]{Au00}, \cite[Theorem 1.12]{Au98}, \cite{MS}).
If $\mathbb{M}_n$  is     the matrix algebra
 of all $n\times n$ matrices over $\mathbb{C}$, then  
 for any element $x\in \mathbb{M}_n$,
\begin{align}\label{eigen-det}
\mbox{{\it a number
$\lambda\in \mathbb{C}$ is an eigenvalue of $x$\,\,\,if and only  if}\,\,\,{\it ${\rm det}(\lambda -x)=0.$}}
\end{align}
In particular, this shows that any unital linear  bijective map on $\mathbb{M}_n$ is spectrum-preserving if and only if it is determinant-preserving. 
However, it is not clear to reduce the problem of  determinant-preserving mappings to that of spectrum-preserving mappings in the setting of $II_1$-factors.
Indeed, there exists a non-zero self-adjoint element $x$ in a type II$_1$-factor such that
${\rm det}(x-\lambda\mathbf{1})>0$ for all complex scalars $\lambda$ (see Example~\ref{exam4} below), which demonstrates that 
\eqref{eigen-det} does not hold in the setting of $II_1$-factors. 
Therefore, the Harris--Kadison Conjecture can not be resolved via an application of Aupetit's result. 
Nevertheless, Theorem \ref{HKcon} above together with Aupetit's result \cite{Au98} indeed yields that the class of all 
unital spectrum-preserving bijections coincides  with that of all unital determinant-preserving bijections on $II_1$-factors.

In \cite{Br}, Brown showed that for every operator $x$ in a type $II_1$ factor $(\cM,\tau)$, there exists a unique compactly supported Borel probability measure $\mu_x$ on $\mathbb{C}$, the Brown measure of $x$, such that for all $\lambda \in \mathbb{C}$, 
$$ \log {\rm det} (x-\lambda {\bf 1})= \tau(\log |x-\lambda {\bf 1}|) = \int_\mathbb{C} \log |t-\lambda |d \mu_x (t). $$
The theory of Brown measures has  been further developed by Haagerup, Larsen, Schultz, 
Dykema  and others \cite{HL00, HS07, HS09, DSZ17},  see Section \ref{sec:proof} for details, and has become a powerful tool in the study of non-self-adjoint operators.

Using the uniqueness of the Brown measure, 
we observe that a linear bijection $\Phi$ between two II$_1$-factors preserves the  Brown measures if and only if it  is  determinant-preserving.
Our strategy to confirm  the Harris--Kadison conjecture
is to show   that a determinant-preserving bijection on a $II_1$-factor must   be a rank metric isometry  (see the following subsection).
To achieve  this, we employ the theory of the Brown measures, and the so-called Haagerup--Schultz projections \cite{HS09}. 

\subsection{Rank metric isometries}

In his seminal work\cite{Neu36,Neu37} (see also \cite{vN60}), von Neumann
introduced 
the notion of a continuous ring, which  is a ring $\cR$ such that its corresponding lattice $L(\cR)$,  the set of all principal right ideals of $\cR,$ ordered by inclusion, 
has a continuous geometry. 
Von Neumann showed that an irreducible regular ring $\cR$ is continuous if and only if it has a rank function
${\rm rk}: \cR \to  [0,1].$ In this case, $\cR$ is complete with respect to the rank metric $(x,y) \to  {\rm rk}(x-y),$
which is induced by the rank function. 
Thus, any irreducible continuous ring has a natural topology, i.e.,  the so-called  {\it rank topology}, which is generated by the rank metric (see \cite{vN60, Sch23}).
The rank metric on regular rings  has important  applications in different areas, e.g., $L_2$-homology for tracial algebras \cite{Th1, Th2} and invariant means on topological groups using rank metrics 
\cite{SchIMRN, Sch23}.

One of the important classes of irreducible continuous rings is the  Murray--von Neumann algebra
$S(\cM)$ associated with a II$_1$-factor equipped with 
 a faithful normal tracial state $\tau_\cM$ (see \cite{Neu58}).
On $S(\cM)$ the rank function can be defined as follows (see \cite{McP, SchIMRN})
\begin{align*}
{\rm rk}(x) & =\tau_\cM(l(x))=\tau_\cM(r(x)),\, x\in S(\cM),
\end{align*}
where $l(x)$ and $r(x)$ are left and right supports of the element $x,$ respectively.

 Recall that a linear
mapping $\Phi:\left(X, \left\|\cdot\right\|_X\right)\to \left(Y, \left\|\cdot\right\|_Y\right)$ between two $F$-spaces is an isometry\cite[p.241]{Rolewicz} if
\begin{align*}
\left\|\Phi(x)-\Phi(y)\right\|_Y & = \left\|x-y\right\|_X,\, \forall x,y \in X.
\end{align*}

The description of commutative and non-commutative $L_p$-isometries has been thoroughly studied (see \cite{HS24, Lamperti, Yeadon} and references therein) since the seminal work of Banach\cite[p.170, Théorème 3]{Banach} and Stone\cite[Theorem 83]{Stone}.
In \cite{BHS24},
the authors  provided a complete description of the limiting case, i.e., 
isometries on  a non-commutative $L_0$-space (i.e., the algebra of all $\tau$-measurable operators affiliated with a semifinite von Neumann algebra equipped with 
 some $F$-norm which is  equivalent to the measure topology. It coincides with $S(\cM)$ when $\cM$ is a   von Neumann algebra with a faithful normal tracial state $\tau_\cM$). 
This extends the classical Banach--Stone theorem and Kadison's theorem for isometries of von Neumann algebras\cite{K51}.
The description \cite[Thorem 4.3 and Corollary 4.4]{BHS24} can be formulated as follows: 
\begin{quote}
the group of all $L_0$-isometries on an $L_0$-space is isomorphic to the semi-direct product of the group of all trace-preserving Jordan $*$-isomorphisms of the $L_0$-space and the group of all unitaries of the  $L_0$-space. 
\end{quote}
However,   the rank topology  and the measure topology on $S(\cM)$ are not equivalent.
The next result provides a complete description of the rank metric isometries of  II$_1$-factors (and the corresponding Murray--von Neumann algebras), which is the crucial step in the proof of Theorem~\ref{HKcon}.

\begin{theorem}\label{isometry}
Let $\cM$ and $\cN$ be von Neumann II$_1$-factors   with faithful normal tracial states $\tau_\cM$ and $\tau_\cN,$
respectively, and let $\Phi$ be a surjective linear  (or conjugate-linear) mapping  from $S(\cM)$ onto $S(\cN).$
Then $\Phi$ is a rank metric isometry if and only if
\begin{align}\label{gen-form}
\Phi(x)=aJ(x)b,\, x\in S(\cM),
\end{align}
where $a, b$ are invertible elements in $S(\cN)$ and $J$ is a   linear (or conjugate-linear) Jordan $*$-isomorphism from $\cM$ onto
$\cN$ (which extends to a Jordan $*$-isomorphism from $S(\cM)$  onto $S(\cN)$).
\end{theorem}

The proof of Theorem~\ref{isometry} 
consists of $4$ main ingredients: 
 (1) the fact that  a rank metric isometry sends  left ideals onto either left or right ideals of $S(\cN)$, see Lemma \ref{leri} below;
(2) a representation of self-adjoint elements of II$_1$-factors as a finite linear combinations of projections~\cite{GP1992};
(3) a characterization of disjointness-preserving additive mappings on matrix algebras~\cite{Bresar07}; 
(4) a general form of ring isomorphism between II$_1$-factors~\cite{AK2020} (see also  \cite{MMori2023}).

\section{Prelinimaries}

\subsection{Murray--von Neumann algebra}\label{sub21}

Let  $H$ be  a complex Hilbert space    and  let   $B(H)$ be   the $\ast$-algebra of all bounded linear operators
on $H.$ Let $\cM$ be a von Neumann algebra in $B(H)$.
As usual we denote by $P(\cM)$ the set of all projections in $\cM.$
Two projections $e, f \in  P(\cM)$ are called \textit{equivalent}  (denoted as $e\sim f$) if there is  an element
$u \in \cM$ such that $u^\ast  u = e$ and $u u^\ast  = f.$
For projections $e, f \in  \cM$
notation $e \precsim  f$ means that there exists a projection $q \in  \cM$ such that
$e\sim q \leq f.$

Recall that a densely defined closed linear operator $x : \textrm{dom}(x) \to  H$
(here the domain $\textrm{dom}(x)$ of $x$ is a dense linear subspace in $H$) is said to be \textit{affiliated} with $\cM$
if $yx \subset  xy$ for all $y$ from the commutant $\cM'$  of the algebra $\cM$\cite{BCLSZ,DPS}.

A linear operator $x$ affiliated with $\cM$ is called \textit{measurable} with respect to $\cM$, if
$e_{(\lambda,\infty)}(|x|)$ is a finite projection for some $\lambda>0.$ Here
$e_{(\lambda,\infty)}(|x|)$ is the  spectral projection of $|x|$ corresponding to the interval $(\lambda, +\infty).$
We denote the set of all measurable operators by $S(\cM)$\cite{BCLSZ,DPS}.

Let $x, y \in  S(\cM).$ It is well known that $x+y$ and
$xy$ are densely-defined and preclosed
operators. Moreover, the closures of 
operators $x + y, xy$ and $x^\ast$  are also in $S(\cM).$
When
equipped with these operations, $S(\cM)$ becomes a unital $\ast$-algebra over $\mathbb{C}$\cite{Segal,BCLSZ,DPS}.
It
is clear that $\cM$  is a $\ast$-subalgebra of $S(\cM).$
In the case of finite von Neumann algebra $\cM$, all operators affiliated with $\cM$ are measurable and the algebra $S(\cM)$ is referred to as the \emph{Murray--von Neumann algebra} associated with $\cM$ (see \cite{KL, BCLSZ}).

Let $\tau$ be a faithful normal finite trace on $\cM.$
Consider the topology  $t_\tau$ of convergence in measure or \textit{measure topology} \cite{BCLSZ,DPS}
on $S(\cM),$ which is defined by
the following neighborhoods of zero:
$$
N(\varepsilon, \delta)=\{x\in S(\cM): \exists \, e\in P(\cM), \, \tau(\mathbf{1}-e)\leq\delta, \, xe\in
\cM, \, \left\|xe\right\|_\cM \leq\varepsilon\},
$$
where $\varepsilon, \delta$
are positive numbers. The pair $(S(\cM), t_\tau)$ is a complete topological $\ast$-algebra.

Let $\cM$ be a finite von Neumann algebra. A $\ast$-subalgebra  $\cA$ of $S(\cM)$ is said to be
{\it regular}, if it is a regular ring in the sense of von Neumann, i.e., if for every
$a\in\cA$ there exists an element  $x\in\cA$ such that $axa=a$ (see \cite{Berber, G79}).
Let $x\in S(\cM)$  and let  $x=v|x|$ be the polar decomposition of the element $x.$
Then $l(x) =v v^\ast$ and  $r(x)=v^\ast v$ are left and right supports of the element  $x,$ respectively.
The projection  $s(x)=l(x)\vee r(x)$ is the support of the element $x.$
It is clear that $r(x)=s(|x|)$ and  $l(x)=s(|x^\ast|).$
There is a unique element $i(x)$ in
$S(\cM)$ such that 
\begin{align}\label{i(x)}
xi(x)=l(x),\ i(x)x=r(x),\ xi(x)x=x,
i(x)l(x)=i(x),\,  r(x)i(x)=i(x).
\end{align}
The element  $i(x)$ is called the \emph{partial inverse} of the element $x.$
In particular, an element 
\begin{align}\label{invertibility}
x\in S(\cM) \,\, \mbox{is invertible if and only if} \,\, r(x)=\mathbf{1}.
\end{align} 

Therefore,   the algebra $S(\cM)$ itself is a regular $*$-algebra   (see also \cite[(VI), page 89]{vN60}, \cite[Theorem~4.3]{Saito}).

Below, we  collect    fundamental properties of left and right projections.

For an element $x\in S(\cM),$  we have  $l(x)\sim r(x).$

For  any $a, b \in S(\cM),$ we have
\begin{align*}
l(axb) & \le l(ax)\sim r(ax)\le r(x)\sim l(x),
\end{align*}
and thus,
\begin{align}\label{lx}
l(axb) & \precsim l(x).
\end{align}
If $a,b$ are invertible, then
\begin{align*}
l(x) & = l\left(a^{-1}(axb)b^{-1}\right) \stackrel{\eqref{lx}}\precsim l\left(axb\right)\stackrel{\eqref{lx}}\precsim l(x),
\end{align*}
which implies
\begin{align}\label{laxb}
l\left(axb\right) & \sim l(x).
\end{align}
In particular, if $a,b$ are invertible, then
\begin{align}\label{llrr}
l\left(xb\right) = l(x) ~\mbox{ and }~ r\left(ax\right) = r(x).
\end{align}
Moreover, note that
\begin{align}\label{a+b}
l(x+y) & \le l(x)\vee l(y).
\end{align}

\subsection{Rank metric on Murray--von Neumann algebras}
Let $\cM$ be a  finite von Neumann algebra  equipped with   a faithful normal finite trace $\tau_\cM$.
Recall that  the so-called rank metric $\rho_{S(\cM)}$ on $S(\cM)$  by setting
\begin{align*}
\rho_{S(\cM)}(x, y)=\tau_\cM( r(x-y))=\tau_\cM(l(x-y)),\,\, x, y\in S(\cM).
\end{align*}
By \cite[Theorem]{McP} (see also \cite[Theorem 17.4]{vN60}), the ring $S(\cM)$ equipped with the metric $\rho_{S(\cM)}$ is a
complete topological $*$-ring.

Define a rank function on $S(\cM)$ as follows
\begin{align}\label{xc}
\left\|x\right\|_S & = \rho_{S(\cM)}(x,0)=\tau_\cM(l(x))=\tau_\cM(r(x)),\, x \in S(\cM).
\end{align}
We have (see \cite[p. 231]{vN60} and \cite[p. 1649]{SchIMRN})
\begin{enumerate}
\item $\left\|x\right\|_S=0$ if and only if $x=0;$
\item $\left\|x+y\right\|_S\le \left\|x\right\|_S + \left\|y\right\|_S$ for all $x, y\in S(\cM);$
\item $\left\|xy\right\|_S \le \min\{\left\|x\right\|_S, \left\|y\right\|_S\}$ for all $x,y \in S(\cM);$
\item $\left\|p+q\right\|_S=\left\|p\right\|_S+\left\|q\right\|_S$ for all $p,q \in P(\cM)$ with $pq=0.$
\end{enumerate}
Note that $\left\|\lambda x\right\|_S=\left\|x\right\|_S$ for all $x\in S(\cM)$ and non-zero $\lambda \in \mathbb{C}.$ Thus
scalar multiplication is not continuous in the rank metric topology,
in particular, $\left(S(\cM), \left\|\cdot\right\|_S\right)$ is not a topological algebra.
The rank metric was first introduced for regular rings in \cite{Neu37}, where its metric properties were established.

\section{Proof of Theorem \ref{isometry}}

Let $\cA$ and $\cB$ be associative $*$-algebras.
Recall that a {\it linear or conjugate-linear} bijection $\Psi:\cA\to \cB$ is called
\begin{itemize}
\item[--] an isomorphism, if $\Psi(xy)=\Psi(x)\Psi(y)$ for all $x,y \in \cA$; 
 \item[--] an anti-isomorphism, if $\Psi(xy)=\Psi(y)\Psi(x) $  for all $x,y \in \cA$;
 \item[--] a Jordan isomorphism, if $\Psi(x^2)=\Psi(x)^2$ for all $x,y \in \cA$. 
\end{itemize}
It is clear that any isomorphism or anti-isomorphism is a Jordan isomorpism.
If $\Psi(x^*)=\Psi(x)^*$ for all $x\in \cA,$ then $\Psi$ is called a 
$*$-isomorphism, $*$-anti-isomorphism and Jordan $*$-isomorphism, respectively.

Let $\Phi$ be a mapping defined by \eqref{gen-form}. Then,  $\Phi$ is invertible and
\begin{align*}
\Phi^{-1}(x)=J^{-1}(a^{-1})J^{-1}(x)J^{-1}(b^{-1}),\, x\in S(\cN)
\end{align*}
or
\begin{align*}
\Phi^{-1}(x)=J^{-1}(b^{-1})J^{-1}(x)J^{-1}(a^{-1}),\, x\in S(\cN),
\end{align*}
depending on $J$ being isomorphism or anti-isomorphism. Furthermore, $\Phi^{-1}$ is also of the  form  \eqref{gen-form}.

Below, we always assume that  $\cM$ and $\cN$ are II$_1$-factors equipped with faithful normal tracial states $\tau_\cM$ and $\tau_\cN$ respectively.  
 Since both $\cM$ and $\cN$ are II$_1$-factors, it follows that any Jordan $*$-isomorphism from
 $S(\cM)$ onto $S(\cN)$ (which maps $\cM$ onto $\cN$) is either $*$-isomorphism or $*$-anti-isomorphism
 (see \cite[Proposition 3.1]{BHS24}). Furthermore, it preserves traces of projections\footnote{ Let $\tau_\cM$ and $\tau_\cN$ be faithful normal tracial states on $\cM,$ $\cN,$ respectively, and $\Psi$ be a Jordan  $*$-isomorphism from $\cM$ onto $\cN.$
Then $\tau_\cN\circ \Psi$ is also tracial state on $\cM.$ By the uniqueness of tracial state, we obtain  $\tau_\cN\circ \Psi=\tau_\cM.$
 Hence, $\Psi$  preserves the traces of projections.}, whence,
\begin{eqnarray*}
\left\|J(x)\right\|_S & \stackrel{\eqref{xc}}{=} & \tau_\cN(l(J(x)))= \tau_\cN(l(J(l(x)x))) = \tau_\cN(l(J(l(x))J(x)))\\
& \le& \tau_\cN(l(J(l(x))))=\tau_\cN(J(l(x)))=\tau_\cM(l(x))=\left\|x\right\|_S, 
\end{eqnarray*}
that is, $\left\|J(x)\right\|_S\le \left\|x\right\|_S$ for all  $x \in S(\cM).$
On the other hand, we have 
\begin{align*}
\left\|x \right\|_S & =\left\|J^{-1}(J(x))\right\|_S \le \left\|J(x)\right\|_S,
\end{align*}
that is,  $\left\|J(x)\right\|_S=\left\|x\right\|_S$ for all $x\in S(\cM).$

Let $a\in S(\cM)$ be an invertible element. Taking into account \eqref{laxb}, we obtain that both the left and right multiplication operators
\begin{align}\label{LR}
L_a(x)=ax~\mbox{ and }~ R_a(x)=xa, ~\forall  x\in S(\cM)
\end{align}
are rank metric isometries of $S(\cM).$
So, each mapping  $\Phi$ of the form~\eqref{gen-form} is a linear rank metric isometry.
It should be noted that if $\Phi$ is a linear rank metric isometry, then the mapping $\Phi^*$ defined as follows
\begin{align}\label{conjugate}
\Phi^*(x) & =\Phi(x)^*,\, x \in S(\cM)
\end{align}
is a conjugate-linear rank metric isometry.

 By~\eqref{invertibility}, an element $x\in S(\cM)$ is invertible if and only if
$r(x)=\mathbf{1}.$ Equivalently, invertibility of
$x$ is characterized by the condition $\|x\|_S=1.$
A fundamental property of rank metric isometries is that they preserve invertibility; that is, they map invertible elements to invertible elements. This yields the following essential property of rank metric isometries, which will be extensively utilized in this section.

\begin{lemma}\label{invert}
Let $\Phi:S(\cM) \to S(\cN)$ be a rank metric isometry and let $x\in S(\cM).$ Then,
$\Phi(x)$ is invertible if and only if $x$ is invertible.
\end{lemma}

We say a bijection \(\phi:P(\mathcal{M})\to P(\mathcal{N})\) is order-preserving if
$
\phi(p) \le \phi(q)
$
for all $p,q \in P(\cM)$ with $p\le q.$

\begin{lemma}\label{exist}
Let $\Phi:S(\cM)\to S(\cN)$ be a  surjective  rank metric isometry.
Then the mappings $\phi, \varphi : P(\cM) \to P(\cN)$ defined by
\begin{align}\label{Psi}
\phi(p)= l\left(\Phi(p)\right),\,  p\in P(\cM),
\end{align}
and
\begin{align}\label{Psir}
\varphi(p)= r\left(\Phi(p)\right),\,  p\in P(\cM),
\end{align}
are order-preserving.
\end{lemma}

\begin{proof} 
Firstly, we observe that  
\begin{align}\label{phip+q}
\phi(p+q) & \le \phi(p)\vee\phi(q).
\end{align}
 for any pair of orthogonal projections $p,q\in P(\cM)$.
Indeed, by  the definition of $\phi$, we have
\begin{align*}
\phi(p+q) & = l\left(\Phi(p+q)\right)=l\left(\Phi(p)+\Phi(q)\right) \stackrel{\eqref{a+b}}{\le} l(\Phi(p))\vee l(\Phi(q))\stackrel{\eqref{Psi}}{=}\phi(p)\vee\phi(q),
\end{align*}
which proves \eqref{phip+q}. 
 
Since $\Phi$ is a rank metric isometry, it follows that 
\begin{align}\label{taup+q}
\tau_\cN(\phi(p+q))  \stackrel{\eqref{Psi}}{=} \tau_\cN\left(l\left(\Phi(p+q)\right)\right)\stackrel{\eqref{xc}}{=}\left\|\Phi(p+q)\right\|_S=\left\|p+q\right\|_S\stackrel{\eqref{xc}}{=}\tau_\cM(p+q)
\end{align}
and
\begin{eqnarray*}
\tau_\cN(\phi(p)\vee \phi(q)) & \stackrel{\eqref{Psi}}{=}&\tau_\cN\left(l(\Phi(p))\vee l(\Phi(q))\right) \\
&\stackrel{\mbox{\tiny \cite[Prop.1.15.9]{DPS}}}{\le}&
\tau_\cN\left(l(\Phi(p))\right)+ \tau_\cN\left(l(\Phi(q))\right) \\
&\stackrel{\eqref{xc}}{=}&  \left\|\Phi(p)\right\|_S+\left\|\Phi(q)\right\|_S\\
& =& \left\| p \right\|_S+\left\| q \right\|_S\\
&\stackrel{\eqref{xc}}{=}&\tau_\cM(p)+\tau_\cM(q)=\tau_\cM(p+q),
\end{eqnarray*}
i.e., 
\begin{align}\label{tp+q}
\tau_\cN(\phi(p)\vee \phi(q)) & \le \tau_\cM(p+q).
\end{align}
Combining \eqref{phip+q}, \eqref{taup+q} and \eqref{tp+q}, we conclude
\begin{align}\label{p+q=}
\tau_\cN(\phi(p+q)) & =\tau_\cN(\phi(p)\vee \phi(q)).
\end{align}
Since $\cN$ is finite, it follows from  \eqref{phip+q} and \eqref{p+q=} that 
\begin{align}\label{p+q+r}
\phi(p+q) & = \phi(p)\vee \phi(q).
\end{align}
In particular,  we have 
\begin{align*}
\phi(p) & \le  \phi(p+q).
\end{align*}

Now,  for any projections 
$e, f \in P(\cM)$   such that $e\le f$, 
we have 
\begin{align*}
\phi(e)   \le  \phi(e+f-e) = \phi(f),
\end{align*}
which proves that $\phi$ is order-preserving.

For the order-preserving property of $\varphi,$ consider the mapping $\Phi^*$ defined by \eqref{conjugate}. For projections $e\le f,$  we have
\begin{align*}
\varphi(e) \stackrel{\eqref{Psir}}{=} r(\Phi(e)) \stackrel{\eqref{conjugate}}{=} l(\Phi^*(e))\stackrel{\eqref{Psi}}{\le} l(\Phi^*(f))=r(\Phi(f))=\varphi(f).
\end{align*}
The proof is complete.
\end{proof}

\begin{remark}\label{p+q} By~\eqref{taup+q} and~\eqref{p+q=}, we obtain that 
$$
\tau_\cN(\phi(p)\vee \phi(q))=\tau_\cM(p)+\tau_\cM(q)=\tau_\cM(\phi(p))+\tau_\cN(\phi(q)).
$$
By the Kaplansky formula\cite[Theorem 6.1.7]{KR2}, we have 
the relation
$\phi(p)\vee \phi(q)-\phi(p)\sim \phi(q)-\phi(p)\wedge \phi(q),$ 
which implies that that
$\tau_\cN(\phi(p)\wedge \phi(q))=0$ (see e.g. \cite[Proposition 1.15.9]{DPS}). 
Consequently, by the faithfulness of $\tau_\cN$, we have  $$\phi(p)\wedge \phi(q)=0$$
for any pair of orthogonal projections $p,q \in P(\cM).$
\end{remark}

Consider orthogonal projections $p,q\in P(\cM)$ with $p+q=\mathbf{1}_\cM.$ The Peirce decomposition\cite{Peirce}  of $S(\cM)$ is given by\footnote{Recall that $yS(\cM)z=\left\{yxz\in S(\cM): x\in S(\cM)\right\},$ where $y, z \in S(\cM)$.}:
\begin{align*}
S(\cM) & = pS(\cM) p\oplus pS(\cM) q\oplus qS(\cM) p\oplus qS(\cM) q,
\end{align*}
in particular, any element $x=px p+px q +qxp+qxq\in S(\cM)$ has a unique matrix representation
\begin{align*}
x & = \begin{pmatrix}
pxp & pxq\\
qxp & qxq
\end{pmatrix}.
\end{align*}

Let $e\in S(\cM)$ be an idempotent, i.e., $e^2=e.$ Below we need the following important property (see for details in \cite{AK2020}):
\begin{align}\label{e+u}
e = p + u,
\end{align}
where 
\begin{align}\label{puq}
u\in pS(\cM)q,\,\, p=l(e)\,\, \mbox{and}\,\,   q=\mathbf{1}_\cM-p.
\end{align}
Observe that  $\mathbf{1}_\cM-e$ is also an idempotent. Using the matrix representation of $e$, we have 
 \begin{align}\label{left-right}
e & = \begin{pmatrix}
p & u\\
0 & 0
\end{pmatrix}~\mbox{ and }~  \mathbf{1}_\cM-e = \begin{pmatrix}
0 & -u\\
0 & q
\end{pmatrix}.
\end{align}
Since $(\mathbf{1}_\cM-e)q\stackrel{\eqref{left-right}}{=}\mathbf{1}_\cM-e,$ it follows that 
$r(\mathbf{1}_\cM-e)\le q.$ On the other hand,
\begin{align*}
q & \stackrel{\eqref{left-right}}{=} q(\mathbf{1}_\cM-e)=q(\mathbf{1}_\cM-e)r(\mathbf{1}_\cM-e)\stackrel{\eqref{left-right}}{=} qr(\mathbf{1}_\cM-e).    
\end{align*}
Thus, $q\le r(\mathbf{1}_\cM-e),$ and hence, $r(\mathbf{1}_\cM-e)=q.$ So,
\begin{align}\label{lr1-}
r(\mathbf{1}_\cM-e)=q=\mathbf{1}_\cM-p=\mathbf{1}_\cM-l(e).
\end{align}

Now let $x\in S(\cM)$ with  $p=l(x)$ and $q=r(x).$
Since $p\sim q$ and $\cM$ is a II$_1$-factor, it follows that
$\mathbf{1}_\cM-p\sim \mathbf{1}_\cM-q$~\cite[Theorem 10.11.12]{Davidson}. Take a partial isometry $w\in \cM$ such that
$\mathbf{1}_\cM-p=w^* w$ and $\mathbf{1}_\cM-q=ww^*.$ Set
$a:=w+i(x),$ where $i(x)$ is the partial inverse of $x.$
We have
\begin{align*}
(w^*+x)(w+i(x)) & = w^*w+w^*i(x)+xw+xi(x)\\
& =\mathbf{1}_\cM-p +w^*ww^* qi(x)+xqww^*w+p\\
& = \mathbf{1}_\cM+ w^*(\mathbf{1}_\cM-q)qi(x)+xq(\mathbf{1}_\cM-q)w=\mathbf{1}_\cM.
\end{align*}
Taking \eqref{lx} into account, we obtain  $l(w+i(x))=r(w+i(x))=\mathbf{1}_\cM,$ and therefore, $a$ is invertible.
Moreover, $xa=xw+xi(x)=xqww^*w+p=p.$ Similarly, we can find an invertible element $b$ such that $bx=q.$
Therefore, there exist invertible elements $a, b$ satisfying: 
\begin{align}\label{invl-r}
xa = l(x)~\mbox{ and }~ bx =r(x).
\end{align}

From now on, we
always  assume that $\Phi:S(\cM)\to S(\cN)$ is a  unital surjective
linear (or conjugate-linear) rank metric isometry (in  Lemmas~\ref{partial}--\ref{zeroproduct}).  In particular,~$\Phi(\mathbf{1}_\cM)=\mathbf{1}_\cN.$

\begin{lemma}\label{partial}
Let $u\in \cM$ be a partial isometry with initial projection $e=uu^*$ and final projection $f=u^*u$. Then, 
$\Phi$ maps $eS(\cM)f$ onto $pS(\cN)q,$ where
$p=l\left(\Phi(u)\right)$ and $q=r\left(\Phi(u)\right).$
\end{lemma}

\begin{proof} 
 {\it Case 1.} We first consider the  special case where  $e=f=u$ and show that 
\begin{center}
$\Phi$ maps $eS(\cM)e$ onto $pS(\cN)q,$
\end{center}
where $p=l\left(\Phi(e)\right)$ and $q=r\left(\Phi(e)\right).$ 
Applying~\eqref{invl-r} to  $x=\Phi(e),$ we find an invertible element $a$ such that
$\Phi(e)a=p.$ 
The mapping  $\Psi=R_a\Phi$ (see \eqref{LR} for the definition of $R_a$) is a rank metric isometry from $S(\cM)$ onto $S(\cN)$ satisfying
\begin{align*}
\Psi(e) & = \Phi(e)a = p= l(\Psi(e))=r(\Psi(e)),
\end{align*}
 here the latter two equalities follow from  that $\Psi(e)=p$ is a projection. Since the mappings $\phi$ and $\varphi$ defined in \eqref{Psi} and  \eqref{Psir} are order-preserving,
it follows that
\begin{align*}
l\left(\Psi(e')\right)=\phi(e')\le \phi(e)=l(\Psi(e))=p,~ r\left(\Psi(e')\right)=\varphi(e')\le \varphi(e)=r(\Psi(e))=p
\end{align*}
for any $e'\in P(e\cM e).$ This shows that
$\Psi$ maps $P(e\cM e)$ into $pS(\cN)p.$

Since  $\cM$ is a II$_1$-factor, it follows from 
\cite[Theorem~1]{GP1992} that  every  self-adjoint element in $e\cM e$ is a real-linear combination of finite number of projections.
Therefore,
$\Psi$ maps $e\cM e$ into $pS(\cN)p.$ Since $e\cM e$ is dense in $S(e\cM e)$
in the rank metric topology, $\Psi$ maps $eS(\cM)e$ into $pS(\cN)p.$
Taking into account that the inverse $\Psi^{-1}$ is also a rank metric isometry, by a similar argument as above we conclude that $\Psi^{-1}$ maps  $pS(\cN)p$ into $eS(\cM)e.$
Thus,
$\Psi$ maps $eS(\cM)e$ onto $pS(\cN)p.$ Since $\Phi=R_{a^{-1}}\Psi,$ it follows that $\Phi$ maps
$eS(\cM)e$ onto $pS(\cN)q.$ 

  {\it Case 2.} Now, we consider the general case when $e\ne f$.
Since $e\sim f$ and $\cM$  is  finite, it follows that
$\mathbf{1}_\cM-e\sim \mathbf{1}_\cM-f$~\cite[Theorem 10.11.12]{Davidson}. Take a partial isometry $v\in \cM$ such that
$\mathbf{1}_\cM-e=vv^*$ and $\mathbf{1}_\cM-f=v^*v.$ Setting
$w=u+v$, which is a unitary operator. Observing that $$fw^*=f(u^*+v^*)= fu^* = u^* u u^* =u^*e =(u^*+v^*)e=w^*e,$$ we see that $R_{w^*}$ maps $eS(\cM)f$ onto $eS(\cM)e.$
Consider the  rank metric isometry $\Theta:=\Phi R_w.$ Since $ew=uu^*(u+v)=uu^*u=u,$ we have $\Theta(e)=\Phi(ew)=\Phi(u).$ 
  Thus,
  $l(\Theta(e))=l(\Phi(u))=p$ and $r(\Theta(e))=r(\Phi(u))=q.$ Therefore, applying Case 1 to $\Theta$,  we see that
$\Theta$ maps $eS(\cM)e$ onto $pS(\cN)q.$ 
Hence, 
$\Phi=\Phi R_wR_{w^*}=\Theta R_{w^*}$ maps $eS(\cM)f$ onto $pS(\cN)q.$
The proof is completed.
\end{proof}

The following lemma characterizes the rank metrics of operators of a certain form. 
 \begin{lemma}\label{pxp=0}
 Let $x=px+q\in S(\cM),$  where $p, q \in P(\cM),$ $pq=0$ and $p+q=\mathbf{1}_\cM.$  Suppose that
 $\tau_\cM(l(x))=\tau_\cM(q).$ Then, $pxp=0.$
\end{lemma}

\begin{proof}
By the definition of $x$, we have 
\begin{align*}
x^*x  = (px+q)^* (px+q) =(x^* p+q) (px+q)=x^*px +q.    
\end{align*}
Then, we have 
\begin{align*}
x^*px+q & =x^*x=x^*x r(x)=(x^*px+q)r(x)=x^*px+qr(x),
\end{align*}
and therefore $q\leq r(x).$
Since $\cM$ is finite and $\tau_\cM(l(x))=\tau_\cM(q),$ it follows that  $q= r(x) \sim l(x).$ 
Hence, 
$pxp=pxr(x)p=pxqp=0.$
The proof is completed.
\end{proof}

Below, we show that $\Phi$ is idempotent-preserving. 
\begin{lemma}\label{idem}
For any idempotent $e\in S(\cM)$, its image $\Phi(e)$ under the mapping $\Phi$ is also an idempotent.
\end{lemma}

\begin{proof} Denote $p=l(\Phi(e))$ and  $q=\mathbf{1}_\cN-p.$
Setting  $f:=\mathbf{1}_\cM-e$, we have
\begin{align}\begin{split}
p\Phi(f)+q &  = p(\mathbf{1}_\cN-\Phi(e))+q = p-p\Phi(e)+q=p+q-\Phi(e)\\
&=\mathbf{1}_\cN-\Phi(e)=\Phi(f).
\end{split}\label{Phif}
\end{align}
Recalling 
that $\Phi$ is a rank metric isometry, we have 
\begin{align*}
\tau_\cN(l(\Phi(f))) & = \left\|\Phi(f)\right\|_S=\left\|f\right\|_S=\tau_\cM(r(f))=\tau_\cM(r(\mathbf{1}_\cM-e))
\overset{\eqref{lr1-}}{=} \tau_\cM(\mathbf{1}_\cM-l(e))\\
& =1-\tau_\cM(l(e))=1-\left\|e\right\|_S
=1-\left\|\Phi(e)\right\|_S= 1-\tau_\cN(p)=\tau_\cN(q).
\end{align*}
This together with~\eqref{Phif} shows that the element $\Phi(f)$ satisfy all conditions of Lemma~\ref{pxp=0}, and hence, we have 
$p\Phi(f)p=0$. Equivalently,  $p=p\Phi(e)p.$
Consequently, we have 
\begin{align*}
\Phi(e) & = l(\Phi(e))\Phi(e)=p\Phi(e)=p\Phi(e)p+p\Phi(e)q=p+p\Phi(e)q,    
\end{align*}
hence it is an idempotent.
This completes the proof.
\end{proof}

Let $\Phi$ be a rank metric isometry.
For an invertible element $a\in S(\cN)$, we 
define $$\Phi^{(a)}:=L_aR_{a^{-1}}\Phi,$$
i.e., $\Phi^{(a)}(x) =a\Phi(x)a^{-1}$ for all $x\in S(\cM).$

Let $e\in P(\cM).$ There exists  an invertible element $a\in S(\cN)$ such that $\Phi^{(a)}(e)$ is a projection.
Indeed, by Lemma~\ref{idem}, $\Phi(e)$ is an idempotent.
Consider the decomposition $\Phi(e)\stackrel{\eqref{e+u}}{=} l(\Phi(e))+v$,  where 
$v\in l(\Phi(e))S(\cM) ({\bf 1}-l(\Phi(e)))$ (see \eqref{puq}). Setting
\begin{align}\label{ale}
a & =  \begin{pmatrix}
l(\Phi(e)) & v\\
0 & \mathbf{1}_\cN-l(\Phi(e))
\end{pmatrix},
\end{align}
we have
 \begin{align*}
a^{-1} & =  \begin{pmatrix}
l(\Phi(e)) & -v\\
0 & \mathbf{1}_\cN-l(\Phi(e))
\end{pmatrix},
\end{align*}
and
\begin{align}\begin{split}
\Phi^{(a)}(e) & =
\begin{pmatrix}
l(\Phi(e)) & v\\
0 & \mathbf{1}_\cN-l(\Phi(e))\end{pmatrix}
\begin{pmatrix}
l(\Phi(e)) & v\\
0 & 0\end{pmatrix}
\begin{pmatrix}
l(\Phi(e)) & -v\\
0 & \mathbf{1}_\cN-l(\Phi(e))\end{pmatrix}  \\
&  =\begin{pmatrix}
l(\Phi(e)) & 0\\
0 & 0\end{pmatrix},
\end{split}\label{vvv}
\end{align}
which shows that $\Phi^{(a)}(e)=l(\Phi(e))$ is a projection.


Before proceeding  to the proof of  Lemma \ref{onto}, we need some preparations. 
Assume that $e, f$ are projections in $\cM$ such that $ef=0,$ $e\sim f,$ $e+f=\mathbf{1}_\cM.$
Let us fix a partial isometry $s\in \cM$ such that
$s^\ast s=e$ and $ss^\ast =f.$ 
We   identify the subsets $eS(\cM)
f,$ $fS(\cM)e,$ and $fS(\cM) f$ with the
$eS(\cM)s^*,$ $s S(\cM)e,$ and $sS(\cM)s^*,$ respectively. 

The Peirce decomposition
    $
    x=x_{11}+x_{12}s^*+s x_{21}+s x_{22}s^\ast\in S(\cM),
    $
    where $x_{ij}\in eS(\cM)e,$ $i,j=1,2,$ provides a representation of  $S(\cM)$ as the matrix algebra
$\mathbb{M}_2(eS(\cM)e):$
\begin{align*}
x\in S(\cM) \to \begin{pmatrix}
    x_{11} & x_{12} \\
    x_{21} & x_{22} \\
  \end{pmatrix}\in \mathbb{M}_2(eS(\cM)e).
\end{align*}

In the proof of   Lemma \ref{onto},   we  utilize the following relation.
\begin{prop} For any 
 $c\in  eS(\cM)e$ with $l(c)=e$ and any $a,b \in eS(\cM)e$, we have
\begin{align}\label{x2p}
l\begin{pmatrix}
0 & a\\
b & c\end{pmatrix} &
\sim
\begin{pmatrix}
l(ai(c)b) & 0\\
0 & e\end{pmatrix}.
\end{align}

\end{prop}
\begin{proof}
    
First, we claim that for any  $x,y\in e S(\cM)e $, we have 
\begin{align}\label{x1p}
l\begin{pmatrix}
x & 0\\
y & c \end{pmatrix}=
\begin{pmatrix}
l(x) & 0 \\
0 & e\end{pmatrix}.
\end{align}
Indeed, by  the first equality in \eqref{llrr},  
Eq.\eqref{x1p}  follows directly from the following relations
 $$\begin{pmatrix}
e & 0\\
y  &  e\end{pmatrix}^{-1}= \begin{pmatrix}
e & 0\\
-y &  e \end{pmatrix},$$
$$
\begin{pmatrix}
e & 0\\
0 & c\end{pmatrix}^{-1}
=\begin{pmatrix}
e & 0\\
0 & i(c)\end{pmatrix} 
$$
and 
\begin{align*}
\begin{pmatrix}
x & 0\\
y & c\end{pmatrix}& =\begin{pmatrix}
x & 0\\
0 &  e \end{pmatrix}
\begin{pmatrix}
e & 0\\
y & e \end{pmatrix}
\begin{pmatrix}
e & 0\\
0 & c \end{pmatrix}. 
\end{align*}

Now, we  present the proof of \eqref{x2p}. 
Observe that
\begin{align}
\begin{pmatrix}\label{ecf}
0 & a\\
b & c \end{pmatrix} &
\begin{pmatrix}
-e &    0  \\
i(c)b & e \end{pmatrix}
=\begin{pmatrix}
ai(c)b & ae\\
-be+ci(c)b & ce\end{pmatrix}
=\begin{pmatrix}
ai(c)b & a\\
0 & c\end{pmatrix}.
\end{align}
On the other hand, we have 
\begin{align}\label{rec}\begin{split}
r\begin{pmatrix}
ai(c)b & a\\
0 & c\end{pmatrix}&=
l\begin{pmatrix}
b^* i(c)^* a^* & 0\\
a^* & c^*\end{pmatrix}\stackrel{\eqref{x1p}}{=}
\begin{pmatrix}
l(b^*i(c)^*a^*) & 0\\
0 & e\end{pmatrix} \\ 
& =\begin{pmatrix}
r(ai(c)b) & 0\\
0 & e\end{pmatrix}\sim \begin{pmatrix}
l(ai(c)b) & 0\\
0 & e\end{pmatrix}.
\end{split}
\end{align}
The element $\begin{pmatrix}
-e & 0\\
ai(c)b & e \end{pmatrix}$ is invertible. Indeed,   $\begin{pmatrix}
-e & 0\\
ai(c)b & e \end{pmatrix}
\begin{pmatrix}
-e & 0\\
ai(c)b & e \end{pmatrix}=\begin{pmatrix}
e & 0\\
0 & e \end{pmatrix}.$
Thus, we have 
\begin{align*}
l\begin{pmatrix}
0 & a\\
b & c\end{pmatrix}
\stackrel{\eqref{ecf},\eqref{llrr}}=
l\begin{pmatrix}
ai(c)b & a\\
0 & c\end{pmatrix}\sim r\begin{pmatrix}
ai(c)b & a\\
0 & c\end{pmatrix}\stackrel{\eqref{rec}}{\sim} \begin{pmatrix}
l(ai(c)b) & 0\\
0 & e\end{pmatrix},
\end{align*}
which  completes the proof of~\eqref{x2p}.
\end{proof}

\begin{lemma}\label{onto}
Let $\Phi$ be a unital  rank  metric isometry, and let $e\in P(\cM)$ be a projection such that $e\sim f:=\mathbf{1}_\cM-e.$ Then,
$$\Phi(eS(\cM)f ) = l(\Phi(e))S(\cN)r(\Phi(f))~ \mbox{ ~\rm or~ }~l(\Phi(f))S(\cN)r(\Phi(e)).$$
\end{lemma}

\begin{proof} 
{\it We first consider the case when 
  $p=\Phi(e)$ is a projection. }
  In this case, 
  $q: =\Phi(f)=\Phi(\mathbf{1}_\cM-e)=\mathbf{1}_\cN-p$ is also  a  projection.  Since $e\sim f$ and $\Phi$ is a rank metric isometry, it follows that 
\begin{align*}
\tau_\cN(p) & = \left\|p\right\|_S=\left\|\Phi(e)\right\|_S=\left\|e\right\|_S=\frac{1}{2} =\left\|f\right\|_S=\left\|\Phi(f)\right\|_S=\left\|q\right\|_S=\tau_\cN(q).    
\end{align*}
Hence, $p\sim q.$  

We consider Peirce decompositions (i.e., $2\times2$-matrix representations) of $S(\cM)$ and $S(\cN)$ with respect to projections
$\{e,f\}$ (in $\cM$) and $\{p, q\}$ (in $\cN$),  respectively.

By Lemma~\ref{partial}, $\Phi$ preserves diagonal blocks
 \begin{align}\label{x12x}
\Phi\begin{pmatrix}
x_{11}& 0\\
0 & x_{22}\end{pmatrix} & = \begin{pmatrix}
\Phi(x_{11}) & 0 \\
0& \Phi(x_{22}) \end{pmatrix},\quad  x_{11}\in eS(\cM)e, x_{22}\in fS(\cM)f.
\end{align}
Since $e \sim f,$ it follows that  there exists  a partial isometry  $u=\begin{pmatrix}
0 & u \\
0 & 0\end{pmatrix}\in e\cM f$  such that $e=l(u),$ $f=r(u).$

Denote $v:=\Phi(u).$ 
Note that for  any non-zero scalar  $\lambda$, the element $e+\lambda u$ is an idempotent.
Lemma~\ref{idem} implies that $\Phi(e+\lambda u)=p+\lambda v$ is also an idempotent.
We have
\begin{align*}
p+\lambda v & = (p+\lambda v)^2= p+\lambda pv+\lambda vp +\lambda^2 v^2.
\end{align*}
Thus, 
\begin{align*}
v & = pv+vp +\lambda v^2.
\end{align*}
Since $\lambda$ is   arbitrarily taken, it follows that $v^2=0 $ and 
\begin{align*}
v & = pv+vp.
\end{align*}
Multiplying both sides of  the above equality by $p$ from the right,  we get $pvp=0.$
 Multiplying both sides of $v  = pv+vp$ by $q$ from the right,    we have 
 $vq=pvq $, i.e., 
 $qvq=0.$
Thus, the Peirce decomposition of $v$ is
\begin{align}\label{v12v}
v =\Phi(u)  = \begin{pmatrix}
0 & v_{12}\\
v_{21} & 0 \end{pmatrix}.
\end{align}

We claim that 
\begin{align}\label{claimeither0}
\mbox{~either~}v_{12}=0\mbox{~or~}v_{21}=0. 
\end{align}

Suppose both 
$v_{12}$ 
and 
$v_{21}$ 
are non-zero. Without loss of generality, we may 
assume
$l(v_{12})\preceq l(v_{21}).$ Since 
$r(v_{12})\sim l(v_{12})\preceq l(v_{21}),$ 
it follows that 
there exists  a unitary $w\in p\cN p$ such that
\begin{align}\label{wl12}
wr(v_{12})w^*\le l(v_{21}).
\end{align}
Therefore, we have 
\begin{align}\label{rwlne0}
r(v_{12})w^*l(v_{21}) & = w^* w r(v_{12})w^*l(v_{21}) \stackrel{\eqref{wl12}}{=} w^*wr(v_{12})w^*=r(v_{12})w^*\neq 0.
\end{align}
Since $r(v_{12})=i(v_{12})v_{12}$ and $l(v_{21})=v_{21}i(v_{21})$ (see~\eqref{i(x)}), it follows from~\eqref{rwlne0} that  
$v_{12}w^*v_{21}\neq 0.$

Setting  
\begin{align}\label{xxxxx}
x=\begin{pmatrix}
0 & u \\
0 & \Phi^{-1}(w) \end{pmatrix},
\end{align}
we have
\begin{align}
\begin{split}\label{phiphi}
\Phi(x)  & =  \Phi\begin{pmatrix}
0 & 0\\
0 & \Phi^{-1}(w)\end{pmatrix}+\Phi\begin{pmatrix}
0 & u\\
0 & 0\end{pmatrix}\\
& \stackrel{\eqref{x12x},\eqref{v12v}}{=} 
\begin{pmatrix}
0 & 0\\
0 & w\end{pmatrix}+\begin{pmatrix}
0 & v_{12}\\
v_{21} & 0\end{pmatrix}
=\begin{pmatrix}
0 & v_{12}\\
v_{21} & w \end{pmatrix}.\end{split}
\end{align}
By \eqref{x2p} (with $b=0$ and $a=u$), we obtain the rank metrics $\left\|x\right\|_S$ and $\left\|\Phi(x)\right\|_S:$ 
\begin{align*}
\frac{1}{2}  = & \tau_\cN(p) =\tau_\cN(l(w))=\left\|w\right\|_S=\left\|\Phi^{-1}w)\right\|_S=\tau_\cM(l(\Phi^{-1}(w))) \stackrel{\eqref{xxxxx},\eqref{x2p}}{=} \tau_\cM(l(x))\\
& =\left\|x\right\|_S
=  \left\|\Phi(x)\right\|_S
\stackrel{\eqref{phiphi}}{=}\left\|
\begin{pmatrix}
0 & v_{12}\\
v_{21} & w \end{pmatrix}\right\|_S
\stackrel{\eqref{x2p}}{=}
\tau_\cN\begin{pmatrix}
l(v_{12}w^* v_{21}) &  0\\
 0 & l(w)\end{pmatrix}\\
 & =
\tau_\cN(l(w))+\tau_\cN(l(v_{12}w^*v_{21}))=\frac{1}{2}+\tau_\cN(l(v_{12}w^*v_{21})).
\end{align*}
This implies $\tau_\cN\left(l(v_{12}w^*v_{21})\right)=0,$ that is,  $v_{12}w^*v_{21}=0,$
which contracts   $v_{12}w^*v_{21}\neq 0.$ This proves \eqref{claimeither0}.

If $v_{21}=0,$ then
$
\Phi(u) = \begin{pmatrix}
0 & v_{12} \\
0 & 0\end{pmatrix}.
$
Since $v_{12}\in pS(\cN)q$ and
\begin{align*}
\tau_\cN(l(v_{12}))=\left\|v_{12}\right\|_S=\left\|\Phi(u)\right\|_S=\left\|u\right\|_S=\tau_\cM(l(u))=\tau_\cM(e)=\frac{1}{2},
\end{align*}
it follows that
$l(v_{12})=p$ and $r(v_{12})=q.$
Lemma~\ref{partial} implies that $\Phi$ maps $eS(\cM)f$ onto $pS(\cN)q.$

If $v_{12}=0,$ then 
$
\Phi^*(u)  = \begin{pmatrix}
0 & v_{21}^* \\
0 & 0\end{pmatrix},
$ where $\Phi^*$ is a rank metric isometry defined as in \eqref{conjugate}.
As we have shown above $\Phi^*$ maps $eS(\cM)f$ onto $pS(\cN)q.$ Hence, $\Phi$ maps $eS(\cM)f$ onto $qS(\cN)p.$

{\it Now let us consider the general case.} Applying~\eqref{vvv}, we find an invertible element $a$ such that
$p=\Phi^{(a)}(e)$ is a projection.

As we have shown above, 
$$\mbox{
$\Phi^{(a)}$ maps $eS(\cM)f$ onto $pS(\cN)q$ or $qS(\cN)p.$}$$
Since $\Phi^{(a)}(x)=a\Phi(x)a^{-1}$ for all $x\in S(\cM),$ it follows that
$\Phi$ maps $eS(\cM)f$ onto either $a^{-1}pS(\cM)qa=l(a^{-1}p)S(\cN)r(qa)$ or $l(a^{-1}q)S(\cN)r(pa).$
Note that
\begin{align*}
l\left(a^{-1}p\right) & =l\left(a^{-1}pa\right)=l\left(a^{-1}\Phi^{(a)}(e)a\right)=l(\Phi(e))
\end{align*}
and
\begin{align*}
r\left(qa\right) & =l\left(a^{-1}qa\right)=l\left(a^{-1}\Phi^{(a)}(f)a\right)=r(\Phi(f)).
\end{align*}
Arguing similarly, we have 
$
l\left(a^{-1}q\right)  =l(\Phi(f))
$ 
and 
$
r\left(pa\right)  =r(\Phi(e)).
$
Therefore,
$\Phi$ maps $eS(\cM)f$ onto either $l(\Phi(e))S(\cM)r(\Phi(f))$ or $l(\Phi(f))S(\cN)r(\Phi(e)).$
The proof is complete.
\end{proof}

\begin{lemma}\label{ephi}
Let  $f\in \cM$ be a projection and let $e\le \mathbf{1}_\cM-f.$ Suppose that  $\Phi(f)$ is a projection. Then the element 
\begin{align*}
a :=  \begin{pmatrix}
l(\Phi(e)) & v\\
0 & \mathbf{1}_\cN-l(\Phi(e))\end{pmatrix},
\end{align*}
defined as in \eqref{ale}, commutes with $\Phi(f).$
\end{lemma}

\begin{proof} Since $e\le \mathbf{1}_\cM-f$ and $\Phi(\mathbf{1}_\cM-f)=\mathbf{1}_\cN-\Phi(f)$ is a projection,
it follows from  Lemma~\ref{exist}    that
 $l(\Phi(e))\le \mathbf{1}_\cN-\Phi(f)$ 
and
$r(\Phi(e))\le \mathbf{1}_\cN-\Phi(f).$ Thus, $\Phi(e)\in (\mathbf{1}_\cN-\Phi(f))S(\cN)(\mathbf{1}_\cN-\Phi(f)).$
Moreover, from the decomposition of the idempotent $\Phi(e)\stackrel{\eqref{e+u}}{=} l(\Phi(e))+v$, we have 
$v\in l(\Phi(e))S(\cN)\left(\mathbf{1}_\cN-\Phi(f)-l(\Phi(e))\right)$  (see~\eqref{puq}).
Therefore,  the Peirce decomposition of the element $a$ with respect to  mutually orthogonal projections $\Phi(f), l(\Phi(e)), \mathbf{1}_\cN-\Phi(f)-l(\Phi(e))$ is in the form of 
\begin{align*}
\begin{pmatrix}
\Phi(f) & 0 & 0 \\
0 & l(\Phi(e)) & v \\
0& 0 & \mathbf{1}_\cN-\Phi(f)-l(\Phi(e))\end{pmatrix}.
\end{align*}
Thus,  $a$ commutes with $\Phi(f),$ which completes the proof.
\end{proof}

Below, we consider the range of $\Phi|_{eS(\cM)}$, where $e$ is a projection.   
\begin{lemma}\label{onto1}
Let $e\in P(\cM)$ be a projection such that $\tau_\cM(e)\le \frac{1}{3}$. Then, 
$$\mbox{
$\Phi$ maps $eS(\cM)$ onto either~$pS(\cM)$~or~$S(\cN)q,$}$$
where $p=l(\Phi(e))$ and $ q=r(\Phi(e)).$
\end{lemma}

\begin{proof} Let $e_1=e $ and $ e_2, e_3$ be projections such that $e_1\sim e_2 \sim e_3,$ $e_2e_3=0$ and $e_2+e_3 \le \mathbf{1}_\cM-e_1.$

{\bf The case when $p_i:=\Phi(e_i),$ $i=1,2,3$, are  projections.}  Since $p_i+p_j=\Phi(e_i+e_j)$ is self-adjoint, it follows from Lemma~\ref{idem} that  it is a projection when $i\ne j$.
By Lemma~\ref{partial},  the restriction of $\Phi$ onto
$(e_i+e_j)S(\cM)(e_i+e_j),$  $ i\neq j$,  is a rank metric isometry from
$(e_i+e_j)S(\cM)(e_i+e_j)$ onto $(p_i+p_j)S(\cN)(p_i+p_j).$ Lemma~\ref{onto} implies that $\Phi$ maps
$e_iS(\cM)e_j$ onto $p_iS(\cN)p_j$ or $p_jS(\cN)p_i,$ where $i, j=1,2,3.$

Assume that  $\Phi$ maps
$e_1S(\cM)e_2$ onto $p_1S(\cN)p_2$\footnote{If $\Phi$ maps
$e_1S(\cM)e_2$ onto $p_2S(\cN)p_1$, then we may 
consider the rank metric isometry $\Phi^*$ defined as in \eqref{conjugate} instead of $\Phi$, which   maps $e_1S(\cM)e_2$ onto $p_1S(\cN)p_2$.} and we claim that  $$\mbox{ $\Phi$ maps
$e_1S(\cM)e_3$ onto $p_1S(\cN)p_3.$}$$

Assume by contradiction that $\Phi$ maps
$e_1S(\cM)e_3$ into a proper subspace of $p_3S(\cN)p_1.$ Take partial isometries
$u_2\in e_1S(\cM)e_2$ and $u_3\in e_1S(\cM)e_3$
such that $l(u_2)=e_1, r(u_2)=e_2,$ $l(u_3)=e_1,$ and $r(u_3)=e_3.$
Set
\begin{align*}
v_i =\Phi(u_i),\quad  i=2,3.
\end{align*}
By assumption,
\begin{align*}
v_2 \in p_1S(\cN)p_2,\, v_3\in p_3S(\cN)p_1.
\end{align*}
Since 
\begin{align*}
\tau_\cN(l(v_i))=\left\|v_i\right\|_S=\left\|\Phi(u_i)\right\|_S=\left\|u_i\right\|_S=\tau_\cM(e_i)=\tau_\cN(p_i),\,~ i=2,3,
\end{align*} it follows that $l(v_2)=p_1$ and $r(v_3)=p_1.$ Set $x=e_1+u_2+u_3.$ 
In particular, 
  $x$ is an idempotent, and,
  by Lemma~\ref{idem},
$\Phi(x)$ is also an idempotent.
However,
$$\Phi(x)=\Phi(e_1+u_2+u_3)=p_1+v_2+v_3$$ and
\begin{align*}
\Phi(x)^2 & = (p_1+v_2+v_3)^2=p_1+v_2+v_3+v_3v_2\neq \Phi(x),
\end{align*}
because $r(v_3)=p_1=l(v_2)$, i.e.,  $v_3v_2\neq 0$ (otherwise,  $p_1=r(v_3)l(v_2)=i(v_3)v_3v_2i(v_2)=0)$.
This contradiction shows  that $\Phi$ must map
$e_1S(\cM)e_3$ onto $p_1S(\cN)p_3.$ By Lemma~\ref{partial}, $\Phi$ maps
$e_1S(\cM)e_1$ onto $p_1S(\cN)p_1.$
Since $e_2, e_3$ are arbitrary orthogonal  subprojections of $\mathbf{1}_\cM-e_1,$
it follows that 
$\Phi$ maps $eS(\cM)$ onto $pS(\cN),$ where $p=\Phi(e)$  and $e=e_1$.
Indeed, any element $x\in eS(\cM)$ can be represented in the following form
\begin{align*}
x & =x_1+x_2+x_3\in e_1S(\cM)e_1+e_1S(\cM)e_2+e_1S(\cM)(\mathbf{1}_\cM-e_1-e_2).
\end{align*}
Since $r(x_3)\sim l(x_3)\le e_1,$
it follows that there exists
  a subprojection $e_3$ of $\mathbf{1}_\cM-e_1-e_2$ such that $e_3\sim e_1$ and $x_3 \in e_1S(\cM)e_3.$ So, 
\begin{align*}
x & =x_1+x_2+x_3\in e_1S(\cM)e_1+e_1S(\cM)e_2+e_1S(\cM)e_3,    
\end{align*}
and hence $\Phi$ maps $eS(\cM)$ onto $pS(\cN).$

{\bf The general case.}
Applying \eqref{vvv} to $\Phi,$ we find an element $a_1$ of the form~\eqref{ale}
 such that $p_1=\Phi^{(a_1)}(e_1)$ is a projection,
 where $\Phi^{(a_1)}$ is defined as in~\eqref{vvv}.

Next, applying \eqref{vvv} to $\Phi_1=\Phi^{(a_1)},$ we find an element $a_2=\begin{pmatrix}
l(\Phi_1(e_2)) & v_2' \\
0 & \mathbf{1}_\cN-l(\Phi_1(e_2))\end{pmatrix}$ of the form~\eqref{ale}
 such that $p_2:=\Phi_1^{(a_2)}(e_2)$ is a projection.
Since $e_2<\mathbf{1}_\cM-e_1$ and $\Phi_1(\mathbf{1}_\cM-e_1)=\mathbf{1}_\cN-p_1,$
applying~Lemma~\ref{ephi} to $\Phi_1$ and $e_2,$ we conclude   that
  $a_2$ commutes with $p_1=\Phi_1(e_1)=\Phi^{(a_1)}(e_1),$
 and therefore
\begin{align*}
\Phi_1^{(a_2)}(e_1)=a_2\Phi_1(e_1)a_2^{-1}=a_2\Phi^{(a_1)}(e_1)a_2^{-1}=a_2p_1a_2^{-1}=p_1.
\end{align*}
Now, applying~\eqref{vvv} to $\Phi_2:=\Phi_1^{(a_2)}$,
we  find an element
$a_3$ such that $p_3:=\Phi_2^{(a_3)}(e_3)$ is a projection.  Since $e_3\le \mathbf{1}_\cM-e_i$ and $\Phi_2(\mathbf{1}_\cM-e_i)=\mathbf{1}_\cN-p_i,$ where $i=1,2,$
it follows from Lemma~\ref{ephi} that
  $a_3$ commutes with both $p_1$ and $p_2.$
 Thus,  
\begin{align*}
\Phi_2^{(a_3)}(e_i) & = a_3\Phi_2(e_i)a_3^{-1}=a_3p_ia_3^{-1}=p_i,   
\end{align*} 
where $i=1,2.$
So, $p_i=\Phi_2^{(a_3)}(e_i)$ ($i=1,2,3$) are all  projections in $\cN.$
Therefore, as we have shown above $\Phi_2^{(a_3)}$ maps $eS(\cM)$ onto $pS(\cN)$. 
By definition, $\Phi_2^{(a_3)}(x)=a\Phi(x)a^{-1}$ for all $x\in S(\cM),$ where $a=a_3a_2a_1.$
Thus, $\Phi$ maps $eS(\cN)$ onto  $a^{-1}pS(\cN)a=a^{-1}pS(\cN)=l(a^{-1}p)S(\cN).$

Observe that
\begin{align*}
l\left(a^{-1}p\right) & =l\left(a^{-1}pa\right)=l\left(a^{-1}\Upsilon^{(a_3)}(e)a\right)=l(\Phi(e)).
\end{align*}
Therefore,  $\Phi$ maps $eS(\cM)$ onto $l(\Phi(e))S(\cN).$ The proof is completed.
\end{proof}

Let $p_1, p_2$ be nontrivial ($\neq 0, \mathbf{1}_\cN$) projections in $\cN.$ 
Observe that 
\begin{align}\label{e1e2}
S(\cN)p_2 \not\subset p_1S(\cN).
\end{align}
Indeed, consider a non-zero partial isometry
$u$ satisfying  $u^*u  \le p_2$ and $uu^* \le \mathbf{1}_\cN - p_1.$  Clearly, $u=uu^*u=up_2 \in S(\cN)p_2$ and $u=uu^*u=(\mathbf{1}_\cN-p_1)u\notin p_1S(\cN).$

Below, we describe the range of the restriction of a unital rank metric isometry to $eS(\cM)$, $e\in P(\cM)$. 
\begin{lemma}\label{onto2}
Let $\Phi$ be a unital rank metric isometry and let $e\in P(\cM)$ be a non-zero projection. 
We have
\begin{enumerate}
\item if $\Phi$ maps $eS(\cM)$ onto $l(\Phi(e))S(\cN),$ then $\Phi$ maps $fS(\cM)$  onto $l(\Phi(f))S(\cN)$ for any
$f\in P(e\cM e);$
\item if $\Phi$ maps $eS(\cM)$ onto $S(\cN)r(\Phi(e)),$ then $\Phi$ maps $fS(\cM)$  onto $S(\cN)r(\Phi(f))$ for any
$f\in P(e\cM e).$
\end{enumerate}
\end{lemma}

\begin{proof} We prove the first statement, and the second one follows  by replacing $\Phi$ with $\Phi^*.$

Let $f$ be a  non-zero sub-projection of $e$ such that $\tau_\cM(f)\le \frac{1}{3}.$ By
Lemma~\ref{onto1}, $\Phi$ maps  $fS(\cM)$ onto either $l(\Phi(f))S(\cN)$ or $S(\cN)r(\Phi(f)).$
Suppose the latter one holds. 
 Since $fS(\cM)\subseteq eS(\cM),$ it  follows that $\Phi(fS(\cM))\subseteq \Phi(eS(\cM)),$ and hence,
$S(\cN)r(\Phi(f))=\Phi(fS(\cM))\subseteq \Phi(eS(\cM))=l(\Phi(e))S(\cN),$
which contracts  \eqref{e1e2}. This contradiction shows that $\Phi$ must map $fS(\cM)$ onto $l(\Phi(f))S(\cN).$

Now, let $f$ be an arbitrary non-zero subprojection of $e.$ As $\cM$ is a II$_1$-factor, we decompose $f$ into mutually orthogonal subprojections
$f_1, f_2, f_3$ with $\tau_\cM(f_i) \le \frac{1}{3}$ for  $i=1,2,3.$
From the preceding argument,  $\Phi$ maps each $f_iS(\cM)$ onto $l(\Phi(f_i))S(\cN)$ for $i=1,2,3.$
By~\eqref{p+q+r},  we have
\begin{align*}
l(\Phi(f))=l(\Phi(f_1))\vee l(\Phi(f_2))\vee l(\Phi(f_3)).
\end{align*}
Remark~\ref{p+q} yields that 
$$l(\Phi(f_i))\wedge l(\Phi(f_j))=0$$  for $i\neq j,$ which yields 
\begin{align*}
l(\Phi(f))S(\cN) & =l(\Phi(f_1))S(\cN)\oplus l(\Phi(f_2))S(\cN)\oplus l(\Phi(f_3))S(\cN),
\end{align*}
where  $\oplus$ denotes  the  direct sum of linear subspaces (not necessarily orthogonal sum).
 Thus, 
\begin{align*}
\Phi(fS(\cM))&=\Phi((f_1+f_2+f_3)S(\cM))
\\
&=\Phi(f_1S(\cM))+\Phi(f_2S(\cM))+\Phi(f_3 S(\cM))= l(\Phi(f))S(\cN),
\end{align*}
which completes the proof.
\end{proof}

The following lemma should be compared with Lemma \ref{onto2}.
\begin{lemma}\label{onto3}
Let $\Phi$ be a unital  rank  metric isometry and let $e\in P(\cM)$ be a non-trivial projection. Then
\begin{enumerate}
\item if $\Phi$ maps $eS(\cM)$ onto $l(\Phi(e))S(\cN),$ then $\Phi$ maps $fS(\cM)$  onto $l(\Phi(f))S(\cN)$ for any
$f\in P(\cM)$ such that $f\ge e$;
\item if $\Phi$ maps $eS(\cM)$ onto $S(\cN)r(\Phi(e)),$ then $\Phi$ maps $fS(\cM)$  onto $S(\cN)r(\Phi(f))$ for any
$f\in P(\cM)$ such that $f\ge e$.
\end{enumerate}
\end{lemma}

\begin{proof} Again it suffices to prove (1).

Let $f$ be a projection in $\cM$ satisfying $f\ge e.$ Take any non-zero   subprojection $g$ of $f$  such that $\tau_\cM(g)<\frac{1}{3}.$
Now, let $e'$ be a  subprojection of $e$ satisfying  $\tau_\cM(g\vee e')\le \frac{1}{3}.$
By Lemma~\ref{onto2},   
\begin{align}\label{elphi}
\Phi\left(e'S(\cM)\right)=l(\Phi(e'))S(\cN).
\end{align}
By Lemma~\ref{onto1}, $\Phi$ maps $(g\vee e')S(\cM)$ onto either $l(\Phi(g\vee e'))S(\cN)$ or $S(\cN)r(\Phi(g\vee e')).$ Suppose the latter one holds. Since $e'S(\cM)\subseteq (g\vee e')S(\cM),$ it follows that 
$\Phi(e'S(\cM))\subseteq \Phi((g\vee e')S(\cM)),$ and hence,
\begin{align*}
l(\Phi(e'))S(\cN) & \stackrel{\eqref{elphi}}{=} \Phi(e'S(\cM)) \subseteq \Phi((g\vee e')S(\cM))=S(\cN)r(\Phi(g\vee e')).    
\end{align*}
Thus,\footnote{For a subset $E\subseteq S(\cN)$, we  denote $E^*=\left\{x^*: x\in E\right\}$.} 
\begin{align*}
S(\cN)l(\Phi(e')) & = \left(l(\Phi(e'))S(\cN)\right)^* \subseteq \left(S(\cN)r(\Phi(g\vee e'))\right)^*=r(\Phi(g\vee e'))S(\cN),    
\end{align*}
which contracts  \eqref{e1e2}. This contradiction shows that $\Phi$ must map $(g\vee e')S(\cM)$ onto $l(\Phi(g\vee e'))S(\cN).$ Applying
Lemma~\ref{onto2} again, since  $g\le g\vee e',$ we conclude that  $\Phi$ maps $gS(\cM)$ onto $l(\Phi(g))S(\cN).$

Next, let $g$ be any subprojection of $f.$
Since $\cM$ is a II$_1$-factor, we may decompose
$g$ as $$g_1+g_2+g_3,$$ 
where $g_1, g_2, g_3$ are mutually orthogonal subprojections of $g$ with
$\tau_\cM(g_i)\le \frac{1}{3}$ for $i=1,2,3.$
Following the reasoning in the proof of Lemma~\ref{onto2}, we conclude that
$\Phi$ maps $gS(\cM)$ onto $l(\Phi(g))S(\cN).$
The proof is completed.
\end{proof}

The next result strengthens Lemma \ref{onto1} above in the setting of unital rank metric isometries. 

\begin{lemma}\label{leri}
Let $\Phi$ be a unital  rank  metric isometry.
Then  $\Phi$  maps
\begin{enumerate}
    \item either  $eS(\cM)$  onto  $l(\Phi(e))S(\cN)$, and $S(\cM)e$  onto   $S(\cN)r(\Phi(e))$;  
    \item or $eS(\cM)$  onto    $S(\cN)r(\Phi(e)),$  and  $S(\cM)e$  onto     $l(\Phi(e))S(\cN)$
\end{enumerate} 
for any projection  $e\in S(\cM).$
\end{lemma}

\begin{proof} We prove the assertion for $eS(\cM)$, as the case for $S(\cM)e$ follows analogously by replacing $\Phi$ with $^*\Phi,$ which is defined by $^*\Phi(x)=\Phi(x^*)^*, x\in S(\cM).$

Let $f\in \cM$ be a arbitrary  non-zero projection  such that $\tau_\cM(f)<\frac{1}{3}.$ By
Lemma~\ref{onto1}, $\Phi$ maps $fS(\cM)$ onto either $l(\Phi(f))S(\cN)$ or $S(\cN)r(\Phi(f)).$

Firstly, we consider the case when  $\Phi$ maps $fS(\cM)$ onto $l(\Phi(f))S(\cN).$

Let $e$ be an arbitrary nontrivial  projection in $\cM.$ Choose a non-zero subprojection $f'$ of $f$ satisfying
$f'\vee e \neq \mathbf{1}_\cM.$
By Lemma~\ref{onto2},
$\Phi$ maps $f'S(\cM)$ onto $l(\Phi(f'))S(\cN).$
Since  $f'\le f' \vee e$, it follows from 
Lemma~\ref{onto3}  that $\Phi$ maps $(f'\vee e)S(\cM)$ onto $l(\Phi(f'\vee e))S(\cN).$ 
Applying Lemma~\ref{onto2} again, we conclude that $\Phi$ maps $eS(\cM)$ onto $l(\Phi(e))S(\cN).$

Now suppose $\Phi$ maps $fS(\cM)$ onto $S(\cN)r(\Phi(f)).$
Then  the rank metric isometry $\Phi^*$ defined by~\eqref{conjugate}, maps
$fS(\cM)$ onto $l(\Phi^*(f))S(\cN),$
because $l(\Phi^*(f))=l(\Phi(f)^*)=r(\Phi(f)).$
Applying Case 1 to  any projection  $e\in S(\cM),$ $\Phi^*$ maps $eS(\cM)$ onto $l(\Phi^*(e))S(\cN). $
Hence, $\Phi$ maps $eS(\cM)$ onto $S(\cN)r(\Phi(e)).$ This completes the proof.
\end{proof}

The next result is the final step before proceeding to  the proof of Theorem \ref{isometry}, which shows that a unital rank metric
isometry or its adjoint is a disjointness-preserving mapping.

\begin{lemma}\label{zeroproduct}
Let $\Phi:S(\cM)\to S(\cN)$ be a unital surjective   rank metric isometry.
For any   $x,y \in S(\cM)$ with  $xy=0,$ either
$
\Phi(x)\Phi(y) = 0
$
or
$
\Phi(y)\Phi(x)  = 0.
$
\end{lemma}

\begin{proof} For any projection $e \in P(\mathcal{M})$, Lemma~\ref{leri} shows
$\Phi$   maps $eS(\cM)$  onto either $l(\Phi(e))S(\cN)$ or $S(\cN)r(\Phi(e)).$

Firstly, we suppose  that $\Phi$   maps $eS(\cM)$  onto $l(\Phi(e))S(\cN).$

Let $x, y\in S(\cM)$ be such that  $xy=0.$ 
By Lemma~\ref{leri}, $\Phi$ maps $l(y)S(\cM)$ and $S(\cM)r(x)$ onto $l(\Phi(l(y)))S(\cM)$
and $S(\cM)r(\Phi(r(x))),$ respectively. 
Thus, $\Phi(y)\in l(\Phi(l(y)))S(\cM)$
and $\Phi(x)\in S(\cM)r(\Phi(r(x))),$ that is,
\begin{align}\label{rrll}
\Phi(x)= \Phi(x)r(\Phi(r(x))),\,\,\, \Phi(y)=l(\Phi(l(y)))\Phi(y).
\end{align}

Since $xy=0,$ it follows that   $r(x)l(y)\stackrel{\eqref{i(x)}}{=}i(x)xyi(y)=0,$ that is, $l(y)\le \mathbf{1}_\cM-r(x).$ By Lemma~\ref{exist} we obtain $l(\Phi(l(y))\le l(\Phi(\mathbf{1}_\cM-r(x))).$
Applying \eqref{lr1-} to the idempotent $\mathbf{1}_\cN-\Phi(r(x))$, we obtain  
\begin{align*}
l(\Phi(l(y)) & \le l(\Phi(\mathbf{1}_\cM-r(x)))=l(\mathbf{1}_\cN-\Phi(r(x)))\stackrel{\eqref{lr1-}}{=}\mathbf{1}_\cN-r(\Phi(r(x))).
\end{align*}
Thus,
$r(\Phi(r(x)))l(\Phi(l(y)))=0.$ Hence,
\begin{align*}
\Phi(x)\Phi(y) & \stackrel{\eqref{rrll}}{=} \Phi(x)r(\Phi(r(x)))l(\Phi(l(x)))\Phi(y)=0,
\end{align*}
which completes the proof of the first case.

Now, suppose $\Phi$   maps $eS(\cM)$  onto $S(\cN)r(\Phi(e)).$ Then $\Phi^*$ maps $eS(\cM)$ onto $l(\Phi^*(e))S(\cN).$
From the preceding argument, $\Phi^*(x)\Phi^*(y)=0.$ 
Thus, $\Phi(y)\Phi(x)=0,$ which completes  the proof.
\end{proof}

Let $\Theta$ be an isomorphism from $S(\cM)$ onto $S(\cN).$ By \cite[Theorem 1.4]{AK2020}, there exist a  $*$-isomorphism
$\Upsilon:\cM\to \cN$ (which extends to a
$*$-isomorphism from
$S(\cM)$ onto $S(\cN)$) and an  invertible element $a\in S(\cN)$ such that
\begin{align}\label{similar}
\Theta(x)=a\Upsilon(x)a^{-1}
\end{align}
for all $x\in S(\cM).$

Now we are in a position to present the proof of Theorem~\ref{isometry}.

\begin{proof}[Proof of Theorem~\ref{isometry}] Let $\Phi:S(\cM)\to S(\cN)$ be a  surjective rank metric isometry. By Lemma~\ref{invert}, $\Phi(\mathbf{1}_\cM)$ is invertible. Setting $\Phi_1:=R_{\Phi(\mathbf{1}_\cM)^{-1}}\Phi,$ we obtain a rank metric isometry
such that $\Phi_1(\mathbf{1}_\cM)=\mathbf{1}_\cN.$

As we have already seen above,   the algebra $S(\cM)$ can be represented as a matrix algebra $$\mathbb{M}_2(eS(\cM)e),$$ where
$e\in P(\cM)$ is a projection with $e\sim \mathbf{1}_\cM-e.$ By Lemma~\ref{zeroproduct}, for any    $x,y \in S(\cM)$ such that $xy=0$ is either
$
\Phi_1(x)\Phi_1(y) = 0
$
or
$
\Phi_1^*(x)\Phi_1^*(y) = 0.
$ Hence, \cite[Corollary~4.2]{Bresar07} implies that either $\Phi_1$ or $\Phi_1^*$ is a homomorphism, that is, either $\Phi_1$ or $\Phi_1^*$ is multiplicative. Since $\Phi_1$ is bijective, it follows that
$\Phi_1$ is either an isomorphism or an anti-isomorphism. By~\eqref{similar}, there exist an $*$-isomorphism
$\Upsilon:S(\cM)\to S(\cN)$ and an  invertible element $a\in S(\cN)$ such that
\begin{align*}
\Theta(x)=a\Upsilon(x)a^{-1}
\end{align*}
for all $x\in S(\cM),$ where $\Theta=\Phi_{1}$ or $\Phi^*_{1}.$  Thus,
$$\mbox{$ \Phi_1(x)=a\Upsilon(x)a^{-1}$ \quad or\quad  $\Phi_1(x)=(a^*)^{-1}\Upsilon(x)^*a^*$}$$ 
for all $x\in S(\cM).$
 Hence, $$\mbox{ 
 $\Phi(x)=a\Upsilon(x)a^{-1}\Phi(\mathbf{1}_{\cM})$ \quad or\quad  $\Phi(x)=(a^*)^{-1}\Upsilon(x)^*a^*\Phi(\mathbf{1}_{\cM})$}$$ for all $x\in S(\cM).$
Therefore, 
in both cases, $\Phi$ has the form~\eqref{gen-form}.
The proof is complete.
\end{proof}

The next result describes  rank metric isometries between two II$_1$-factors (the bounded parts of Murray--von Neumann algebras).

\begin{corollary}\label{isometryvNa}
Let $\cM$ and $\cN$ be von Neumann II$_1$-factors   with faithful normal tracial states $\tau_\cM$ and $\tau_\cN,$
respectively. Suppose  that $\Phi:\cM \to \cN$  is a  linear or conjugate-linear bijection.
 Then $\Phi$ is a rank metric isometry from $\cM$ onto $\cN$ if and only if
\begin{align*}
\Phi(x)=aJ(x)b,\, x\in \cM,
\end{align*}
where $a, b$ are invertible elements in $\cN$   and $J$ is a linear or conjugate-linear Jordan $*$-isomorphism from $\cM$ onto
$\cN.$ If $\Phi$ is unital, it is either an isomorphism or anti-isomorphism.
\end{corollary}

\begin{proof} Since $\cM$ is dense in $S(\cM)$  with respect to the rank topology (see \cite[Proposition]{McP}, it follows
that $\Phi$ can be uniquely extended to a  surjective rank metric isometry from $S(\cM)$ onto $S(\cN),$
which we still denote by $\Phi.$
By Theorem~\ref{isometry}, $\Phi$ is of the form \eqref{gen-form}.

We first show that $a, b \in \cN.$  Let $a=u|a|$ and $b=v|b|$ be the polar decompositions of $a$ and $ b$,
respectively.
Select  $0<\delta<\epsilon$ such that $p=e_{(\delta, \epsilon)}(v|b|v^\ast)$ is non-zero. Since
$\lim\limits_{\lambda\uparrow +\infty}e_{(\lambda, +\infty)}(|a|)=0,$ we may  choose
$\varepsilon>0$ such that $$q:=e_{(\varepsilon, +\infty)}(|a|)\preceq p.$$
Choose a partial isometry $w\in \cN$ satisfying $q=ww^\ast$ and $p_0=w^\ast w \le p.$
Since $J$ maps $\cM$ onto $\cN,$ it follows that the element $x=J^{-1}(w)$ lies in $\cM.$
Since $\Phi(x)\in \cN$ and
\begin{align*}
\Phi(x) & = aJ(x)b =awb=a w(w^\ast w)b=a w p_0b,
\end{align*}
it follows that
\begin{align}\label{awp}
awp_0b\in \cN.
\end{align}
Recalling that  $p_0 \le p=e_{(\delta,\epsilon)}(v|b|v^\ast),$ we derive
\begin{align*}
p_0bv^\ast   & = p_0 p b v^\ast = p_0 e_{(\delta,\epsilon)}(v|b|v^\ast)v|b|v^\ast,
\end{align*}
which implies that  $\delta^2 p_0\le p_0 (v|b|v^*)^2 p_0  =  p_0 bb^*  p_0 =p_0 (v|b|v^*)^2 p_0 \le \epsilon^2 p_0.$ Thus, $p_0bv^\ast$ is invertible in $p_0 \cN p_0.$
Consequently,
\begin{align}\label{aw}
aw & = aw p_0bv^\ast \cdot i(p_0bv^\ast) \stackrel{\eqref{awp}}{\in} \cN.
\end{align}
Since $|a|=u^*a$ and $e_{(\varepsilon, +\infty)}(|a|)=ww^*$, it follows that
\[
|a|e_{(\varepsilon, +\infty)}(|a|)= u^\ast aw w^\ast  \stackrel{\eqref{aw}}\in \cN,
\]
which yields that  $|a|=|a|e_{[0, \varepsilon]}(|a|)+|a|e_{(\varepsilon, +\infty)}(|a|)\in \cN.$ Hence, $a\in \cN.$ Similarly, we deduce that  $b\in \cN.$

The inverse $\Phi^{-1}$ is also a surjective rank metric isometry. Since  $J$ is an isomorphism or anti-isomorphism, it follows that 
\begin{align*}
\Phi^{-1}(x)=J^{-1}(a^{-1})J^{-1}(x)J^{-1}(b^{-1}),\quad x\in S(\cN)
\end{align*}
or
\begin{align*}
\Phi^{-1}(x)=J^{-1}(b^{-1})J^{-1}(x)J^{-1}(a^{-1}),\quad x\in S(\cN).
\end{align*}
The preceding argument shows that  $J^{-1}(a^{-1}), J^{-1}(b^{-1})\in \cM,$ i.e.,  $a^{-1}, b^{-1}\in \cN.$
 This completes the proof.
\end{proof}

\section{Proof of Theorem \ref{HKcon}}\label{sec:proof}

Let $\cM$ be a von Neumann  factor of type II$_1$ with a faithful normal tracial state $\tau_\cM.$
The Fuglede--Kadison determinant \cite{FK52, HS07, HS09, DSZ17}, is the
multiplicative map ${\rm det} : \cM \to  [0, \infty)$,  defined by
\begin{align}\label{det}
{\rm det}(x) & =\lim\limits_{\varepsilon\downarrow 0}{\rm exp}\left(\tau_\cM({\rm log}(|x|+\varepsilon\mathbf{1}_\cM))\right),\, x\in \cM.
\end{align}
Fuglede and Kadison characterized the determinant in a type II$_1$-factor
using its algebraic properties. Specifically, they proved~\cite[Theorem 3]{FK52} that 
if a  
numerical-valued function
$\Delta$ defined on $\cM_{inv}=\left\{x\in \cM: x^{-1}\in \cM\right\}$ (the group of invertible elements in $\cM$) satisfying the following properties:
\begin{enumerate}
\item $\Delta(xy) = \Delta(x)\Delta(y)$ for all $x,y \in \cM_{inv};$
\item $\Delta(x^*) = \Delta(x)$ for any $x\in \cM_{inv};$
\item $\Delta(\lambda \mathbf{1}_\cM) = \lambda$ for some positive scalar $\lambda \neq 1;$
\item $\Delta(x) \le 1,$ if $0 \le x \le \mathbf{1}_\cM$  and $x\in \cM_{inv},$
\end{enumerate}
then it coincides with the restriction ${\rm det}|_{\cM_{inv}}$  of the determinant ${\rm det}$ 
defined in~\eqref{det}.

The Fuglede–Kadison determinant extends naturally (as in \eqref{det}) to the
 space  $\mathcal{L}_{\rm log}(\cM, \tau_\cM)$ of all
$x \in  S(\cM)$ such that
\begin{align*}
\tau_\cM\left({\rm log}^+|x|\right)=\int\limits_0^\infty{\rm log}^+\lambda\, d\tau_\cM(e_\lambda(|x|))<\infty,
\end{align*}
where ${\rm log}^+(t) = \max\{{\rm log}\,t, 0\}.$ 
The space $\mathcal{L}_{\rm log}(\cM, \tau_\cM)$ equipped with the $F$-norm $\left\|x\right\|_{{\rm log}}=\tau_\cM({\rm log}(\mathbf{1}_\cM +|x|)),$ $x \in \mathcal{L}_{\rm log}(\cM, \tau_\cM)$,  forms a complete topological $*$-algebra~\cite{DSZ16}.
See \cite{HS07} for a systematic  development of the 
Fuglede--Kadison determinants on $\mathcal{L}_{\log }(\cM,\tau_\cM)$.

Below, we recall   some necessary properties of the Brown measure (see for details \cite{Br, HL00, HS07, HS09,Schultz06}).

For a fixed element $x \in  \cL_{\rm log}(\cM,\tau_\cM)$,
the function~\cite[Theorem 2.7]{HS07} (see also \cite[p.20]{HS09})
\begin{align*}
L(x-\lambda\mathbf{1}_\cM): \lambda \mapsto {\rm log}\left({\rm det}(x - \lambda\mathbf{1}_\cM)\right)
\end{align*}
is subharmonic in $\mathbb{C},$ and its Laplacian
\begin{align}\label{bro}
d\mu_x(\lambda) & =\frac{1}{2\pi}\nabla^2 {\rm log}\left({\rm det}\left(x -\lambda\mathbf{1}_\cM\right)\right)d\lambda
\end{align}
(taken in the distribution sense) defines a probability measure $\mu_x$ on $\mathbb{C},$ called {\it the Brown measure
of $x.$} 
The Brown measure is the unique probability measure on $(\mathbb{C},\mathcal{B}(\mathbb{C}))$ satisfying
\begin{align*}
\int\limits_{\mathbb{C}}{\rm log}^+|t|d\mu_x(t)<\infty,
\end{align*} 
and 
\begin{align}\label{Ldet}
L(x-\lambda\mathbf{1}_\cM) =\int\limits_{\mathbb{C}}{\rm log}|t-\lambda| d\mu_x(t)
\end{align}
for all complex numbers $\lambda.$

If $x\in L_{\rm log}(\cM, \tau_\cM)$ is a normal operator, then we have 
\begin{align}\label{normalBrown}
\mu_x=\tau_\cM\circ {\rm E}_x,
\end{align}
where ${\rm E}_x : \mathcal{B}(\mathbb{C}) \to  P(\cM)$ (with $\mathcal{B}(\mathbb{C})$ denoting  the Borel $\sigma$-algebra of $\mathbb{C}$) is the projection-valued measure on $(\mathbb{C},\mathcal{B}(\mathbb{C}))$ in the spectral resolution of $x,$ that is,
\begin{align*}
x & = \int\limits_{\sigma(x)}\lambda d e_\lambda(x),\,\,   {\rm E}_x: A\in \mathcal{B}(\mathbb{C}) \mapsto e_A(x), 
\end{align*}
where $e_A(x)$ is the spectral projection corresponding to
$A\in \mathcal{B}(\mathbb{C}).$

By \cite[Proposition 2.17]{HS07},  the support ${\rm supp}(\mu_x)\subseteq \sigma(x)$  for an arbitrary element $x\in \cM$.

For  an element $x$ in the type I$_n$-factor $\mathbb{M}_n(\mathbb{C}),$ the spectrum $\sigma(x)$
consists of the roots (counting multiplicities)  of the characteristic polynomial
\begin{align}\label{plambda}
P(\lambda) & ={\rm det}(\lambda\mathbf{1}_n-x)=(\lambda-\lambda_1)\cdots(\lambda-\lambda_n),
\end{align}
where $\lambda_1, \ldots, \lambda_n$ are the roots repeated according to algebraic multiplicity. 
In this case, 
the Fuglede--Kadison determinant and the Brown
measure  defined for $x$ are 
\begin{align*}
{\rm det}(x) & = |{\rm det}_n(x)|^{\frac{1}{n}}    
\end{align*} 
and
\begin{align*}
\mu_x & = \frac{1}{n}(\delta_{\lambda_1}+\cdots+\delta_{\lambda_n}),
\end{align*}
respectively, 
where ${\rm det}_n(x)$ denotes the classical determinant of $n\times n$-matrix and $\delta_{\lambda_i}$ is the Dirac measure at $\lambda_i,$ $i=1, \ldots, n$\cite[p.20]{HS09}.

We now state two key properties of the Brown measure due to Haagerup and Schultz  \cite{HS07, HS09,Schultz06}.

The following theorem, the main result (Theorem 1.1) of \cite{HS09}, plays a central role in our analysis. 
It characterizes those projections that decompose an operator $x$ with respect to  the Brown measure.

\begin{theorem}\label{HSprojection}\cite[Theorem 4]{DSZ15}\cite[Theorem 1.1]{HS09} Let $x\in \cM.$
For any Borel set $B \subseteq  \mathbb{C},$ there exists a unique projection
$p = p(x, B)$ such that
\begin{itemize}
\item[(i)] $xp = pxp,$  
\item[(ii)] $\tau_\cM(p) =\mu_x(B),$
\item[(iii)] when $p \neq 0,$ considering $xp$ as an element of $p\cM p,$ its Brown measure
$\mu_{xp}$ is concentrated in $B,$  
\item[(iv)]  when $p \neq \mathbf{1}_\cM,$   considering $(\mathbf{1}_\cM-p)x$ as an element of $(\mathbf{1}_\cM-p)\cM (\mathbf{1}_\cM-p),$ 
its Brown measure
is concentrated in $\mathbb{C}\setminus B.$ 
\end{itemize}
\end{theorem}

The projection $p(x, B)$ is called the {\it Haagerup--Schultz projection} of $x$
associated with $B$  (see e.g. \cite[p. 98]{DSZ15}).

The following result  gives a decomposition of the Brown measure $\mu_x$ of $x$ via $x$-invariant projections $p$, i.e., $pxp=xp$ (see \cite[Proposition 6.5]{HS09}).

\begin{prop}\label{HSdecomposition} Let $x\in L_{\rm log}(\cM, \tau_\cM)$ and let $p$ be a nontrivial  $x$-invariant projection
(in particular, we   may write $x=\begin{pmatrix}
xp & pxq\\
0 & qxq \end{pmatrix},$ $q=\mathbf{1}_\cM-p$).
Then, 
\begin{align*}
\mu_x &  = \tau_\cM(p)\mu_{xp}+\tau_\cM(q)\mu_{qx},
\end{align*}
where $\mu_x$  denotes the Brown measure of $x,$  and $xp$ and $qx$ are considered
as elements of $\cL_{\rm log}\left(p\cM p, \frac{1}{\tau_\cM(p)}\tau|_{p\cM p}\right)$ and $\cL_{\rm log}\left(q\cM q,\frac{1}{\tau_\cM(q)}\tau|_{q\cM q}\right),$ respectively.
\end{prop}

An element $x\in \cM$ is said to be s.o.t.-quasinilpotent if the sequence $\left\{\left((x^*)^nx^n\right)^{\frac{1}{n}}\right\}_{n\ge 1}$ 
converges in the strong operator topology to $0$ --- by Corollary~2.7 in \cite{HS09}, this is equivalent to the Brown measure of $x$ being concentrated at $0$ (see also \cite[Definition 1.4]{DFS14}).

Let us present a simple example where bounded and unbounded operators in a II$_1$-factor share the same Brown measures.

Let $c\in \cM$ be a positive element with nontrivial support $p=s(c)$ and consider
\begin{itemize}
\item  $a$  be an s.o.t.-quasinilpotent element in $q\cM q,$ where $q=\mathbf{1}_\cM-p;$
 \item $b\in pL_{\rm log}(\cM, \tau_\cM)q$ be an arbitrary element, possibly unbounded.
 \end{itemize}
By Proposition~\ref{HSdecomposition} the operators
\begin{align*}
x=\begin{pmatrix}
c & 0 \\
0 & 0 \end{pmatrix},\,\,\, y=\begin{pmatrix}
c & b \\
0 & a \end{pmatrix}
\end{align*}
have same Brown measures $\mu_x=\mu_y=\tau_\cM(p)\mu_c+(1-\tau_\cM(p))\delta_0,$ where $\delta_0$ is the Dirac measure at $0.$

Before proceeding to the proof of Theorem~\ref{HKcon}, we need one more auxiliary result, which links the Brown measure to the rank metric. 
\begin{lemma}\label{sameBrown} Let $x\in \cM$ be a non-zero positive element and let $y\in L_{\rm log}(\cN, \tau_\cN)$ be a non-zero element such that the  Brown measures of $x$ and $y$ coincide, i.e., $\mu_x=\mu_y.$ Then, $\left\|x\right\|_S\le \left\|y\right\|_S.$
\end{lemma}

\begin{proof} Let $p=l(y)$ be the left support of $y.$

{\bf Case 1}. If $p=\mathbf{1}_\cN,$ then we have
\begin{align*}
\left\|y\right\|_S & =\tau_\cN(\mathbf{1}_\cN)=1 \ge \tau_\cM(s(x))=\left\|x\right\|_S.
\end{align*}

{\bf Case 2}. Assume $p\neq \mathbf{1}_\cN.$ We decompose $y$ as
\begin{align*}
y=pyp+pyq=\begin{pmatrix}
pyp & pyq\\
0 & 0\end{pmatrix}, ~q=\mathbf{1}_\cN-p,
\end{align*}
which shows  $p$ is $y$-invariant.
Since $qy=0$ implies that  $\mu_{qy}=\delta_0,$ where $\delta_0$ is the Dirac measure at  $0,$  it follows from 
  Proposition~\ref{HSdecomposition} that  the Brown measure of $y$ is
\begin{align*}
\mu_y=\tau_\cN(p)\mu_{yp}+\tau_\cN(q)\delta_0.
\end{align*}

By the  assumption that  $\mu_x=\mu_y,$ we have
\begin{align}\label{muxdelta}
\mu_x=\tau_\cN(p)\mu_{yp}+\tau_\cN(q)\delta_0.
\end{align}

Let $e=p(x, \mathbb{C}\setminus \{0\})$ denote  the Haagerup--Schultz projection of $x$ associated with
$\mathbb{C}\setminus \{0\}.$ 
By Theorem~\ref{HSprojection}(ii), we have $\tau_\cM(e) =\mu_x(\mathbb{C}\setminus \{0\}).$  
We claim that $e\ne 0$. Assume by contradiction that  $e=0.$
Hence, 
$\mu_x(\mathbb{C}\setminus \{0\})=0$,  that is,  $\mu_x(\{0\})=1$ and $\mu_x=\delta_0.$ By~\cite[Corollary~2.7]{HS09}, $x$ is s.o.t.-quasinilpotent. Then $x=0,$\footnote{ Any operator  that is simultaneously normal and s.o.t.-quasinilpotent must necessarily be zero.
In fact, assume $a$ is normal and $a$ is s.o.t.-quasinilpotent. Then by~\cite[Corollary~2.7]{HS09}, we have $|a|^2=a^*a=\left((a^*)^n a^n\right)^{\frac{1}{n}}\stackrel{so}{\longrightarrow}0$ as $n\to \infty,$ and hence, $a=0.$} that contradicts the assumption that $x$ is non-zero. 
Moreover, 
\begin{align}\label{tauxtauy}
\tau_\cM(e) & =\mu_x(\mathbb{C}\setminus \{0\})\stackrel{\eqref{muxdelta}}{=}\tau_\cN(p)\mu_{yp}(\mathbb{C}\setminus \{0\})\le
\tau_\cN(p)\mu_{yp}(\mathbb{C})=\tau_\cN(p).
\end{align}
By Theorem~\ref{HSprojection}, $e$ is $x$-invariant. Thus, we may decompose  $x$  as
$$x=\begin{pmatrix}
exe & exf\\
0 & fxf \end{pmatrix},$$ where $f=\mathbf{1}_\cM-e.$
Since $x$ is positive,  we have $exf=0$ and $f x f$ is positive.
Applying Theorem~\ref{HSprojection}(iv) again, the Brown measure $\mu_{fxf}$ is concentrated on $0,$
that is, $fxf$ is s.o.t.-quasinilpotent~\cite[Corollary~2.7]{HS09}.
We have $fxf=0$.
Hence,
$x=\begin{pmatrix}
exe & 0 \\
0 & 0 \end{pmatrix},$
which implies that   $s(x)=s(exe)\le e.$  Therefore
\begin{align*}
\left\|x\right\|_S & =\tau_\cM\left(s(x)\right) \le \tau_\cM(e) \stackrel{\eqref{tauxtauy}}{\le } \tau_\cN\left(p\right)=\tau_\cN(l(y))=\left\|y\right\|_S.
\end{align*}
This completes the proof.
\end{proof}

Now we are in a position to present the proof of Theorem~\ref{HKcon}.

\begin{proof}[Proof of Theorem~\ref{HKcon}] 
($\Leftarrow$). Since any Jordan $*$-isomorphism between two II$_1$-factors is trace-preserving, it follows that $\Phi$ is also determinant-preserving (see definition~\eqref{det}). 
By Corollary \ref{isometryvNa}, $\Phi$ is of the form \eqref{precise form}.  
  By multiplicity of the determinant, the mapping 
  $x\in \cN \mapsto axa^{-1}\in \cN,$ where $a$ is invertible element in $\cN$ (see Corollary \ref{isometryvNa}), is also determinant-preserving. In other words,  any Jordan isomorphism is a determinant-preserving mapping.

($\Rightarrow$). 
Let $x\in \cM$ be a non-zero element, and consider its  polar decomposition
$x=u|x|$ (the partial isometry $u$ can be chosen to be  unitary since  $\cM$ is finite).
Since $\Phi$ is  determinant-preserving, it follows that 
\begin{align*}
 {\rm det}(x-\lambda u) = {\rm det}(\Phi(x-\lambda u) )  ={\rm det}(\Phi(x)-\lambda \Phi(u)),\, \lambda\in \mathbb{C}
\end{align*}
and
\begin{align*}
{\rm det}(\Phi(u))={\rm det}(u)=\sqrt{{\rm det}(u^* u)}=\sqrt{{\rm det}(\mathbf{1}_\cM)}=1.
\end{align*}
Since ${\rm det}(\Phi(u))=1,$ it follows from  \cite[Lemma 2.3]{HS07} that $\Phi(u)^{-1}\in \mathcal{L}_{\rm log}(\cN, \tau_\cN),$ and
hence, $\Phi(x)\Phi(u)^{-1}\in \mathcal{L}_{\rm log}(\cN, \tau_\cN).$
Using the multiplicativity of the determinant, we compute:
\begin{align*}
{\rm det}(x-u\lambda) & ={\rm det}(u(u^*x-\lambda\mathbf{1}_\cM))={\rm det}(u){\rm det}(|x|-\lambda\mathbf{1}_\cM)={\rm det}(|x|-\lambda\mathbf{1}_\cM)
\end{align*}
and similarly,
\begin{align*}
{\rm det}(\Phi(x)-\lambda \Phi(u))= {\rm det}\left(\Phi(x)\Phi(u)^{-1}-\lambda\mathbf{1}_\cN\right){\det}(\Phi(u))={\rm det}\left(\Phi(x)\Phi(u)^{-1}-\lambda\mathbf{1}_\cN\right).
\end{align*}
By the determinant-preserving property of $\Phi$, we have  
\begin{align*}
{\rm det}(|x|-\lambda \mathbf{1}_\cM)={\rm det}\left(\Phi(x)\Phi(u)^{-1}-\lambda \mathbf{1}_\cN\right).
\end{align*}
Taking the logarithm of determinants, we obtain
\begin{align*}
L(|x|-\lambda\mathbf{1}_\cM) & ={\rm log}\left({\rm det}(|x|-\lambda\mathbf{1}_\cM)\right)\\
&={\rm log}\left({\rm det}(\Phi(x)\Phi(u)^{-1}-\lambda\mathbf{1}_\cN)\right)=L(\Phi(x)\Phi(u)^{-1}-\lambda\mathbf{1}_\cN).
\end{align*}
By the uniqueness of the Brown measure (see also \eqref{bro},~\eqref{Ldet}), it follows that
$\mu_{|x|}=\mu_{\Phi(x)\Phi(u)^{-1}}.$ By Lemma~\ref{sameBrown},
we have
 \begin{align*}
 \left\|x\right\|_S & = \left\||x|\right\|_S \stackrel{{\rm Lem.}\ref{sameBrown}}{\le}\left\|\Phi(x)\Phi(u)^{-1}\right\|_S\stackrel{\eqref{llrr}}{=}\left\|\Phi(x)\right\|_S.
 \end{align*}
Thus,    $\left\|x\right\|_S\le \left\|\Phi(x)\right\|_S.$ 
On the other hand, since $\Phi^{-1}$ also preserves determinants,
 the inverse inequality $\left\|x\right\|_S\ge \left\|\Phi(x)\right\|_S$ holds. Thus,  $$\left\|x\right\|_S=\left\|\Phi(x)\right\|_S$$ for all $x\in \cM,$ proving that
  $\Phi$ is a rank metric isometry. By Corollary~\ref{isometryvNa}, $\Phi$ must be either an isomorphism or an anti-isomorphism.   This completes the proof.
\end{proof}

From the uniqueness of the Brown measure and Eq.\eqref{Ldet}, it follows that a linear bijection $\Phi$ between two II$_1$-factors is Brown-measure-preserving if and only if it is determinant-preserving.\footnote{ Indeed, a unital linear bijection $\Phi:\cM \to \cN$ is determinant-preserving if and only if ${\rm det}(x-\lambda \mathbf{1}_\cM) = {\rm det}(\Phi(x)-\lambda \mathbf{1}_\cN)$ for all $x\in \cM$ and $\lambda \in \mathbb{C},$ which is equivalent $\mu_x=\mu_{\Phi(x)}$ (see \eqref{bro},~\eqref{Ldet}).} Therefore, Theorem~\ref{HKcon} implies the following result.

\begin{corollary}\label{Bro-mp}
Let $\cM$ and $\cN$ be factors of type II$_1$  and let  $\Phi$  be a unital linear bijection of
$\cM$ onto $\cN$.  Then, $\Phi$ preserves the Brown measures  if and
only if it is an  isomorphism or an anti-isomorphism.

\end{corollary}

\section{Remarks and examples}

\begin{remark} It is important to note that  the equality of Brown measures
$\mu_x=\mu_y$ does not imply $\mu_{|x|}=\mu_{|y|}.$
Indeed, consider a non-zero partial isometry $u\in \cM$ such that $u^2=0$.  
Let $x=u$ and $y=\lambda u,$ where $\lambda$ is a non-zero complex scalar with $|\lambda|\neq 1.$ 
Since $u^2=0,$ it follows that
\begin{align*}
\mu_x =\delta_0 =\mu_y .
\end{align*} 
On the other hand,
\begin{align*}
\mu_{|x|} & =\mu_{u^*u} =\tau_\cM(u^*u)\delta_1+(1-\tau_\cM(u^*u))\delta_0,\\
\mu_{|y|} & =\mu_{|\lambda|u^*u} =\tau_\cM(u^*u)\delta_{|\lambda|}+(1-\tau_\cM(u^*u))\delta_0.
\end{align*}
Therefore, in the proof of Theorem~\ref{HKcon}, we   work with $|x|$ and $\Phi(x)\Phi(u)^{-1}$ instead of $x$ and $\Phi(x)$. 
\end{remark}

\begin{remark} Let $\Phi: \mathbb{M}_n \to \mathbb{M}_n$ be a linear determinant-preserving bijection. Then  $\Phi$ preserves the characteristic polynomial of matrices as given by \eqref{plambda}.
Consequently, in this setting, a determinant-preserving bijection coincides precisely with a Brown measure-preserving map. In the proof of Theorem~\ref{HKcon}, we  make use of   the absolute values of matrix elements and demonstrate that any determinant-preserving bijection is also rank-preserving. 
Our approach remains novel even within the framework of matrix algebras.
\end{remark}

\begin{example}\label{exam4} Let $\cM$ be a   type I$_n$-factor, that is, the matrix algebra
$\mathbb{M}_n$ of all $n\times n$ matrices over the complex numbers
$\mathbb{C}$ and let $x\in \mathbb{M}_n.$ 
As mentioned in the Introduction, 
A complex number  $\lambda $ is an eigenvalue of $x$ if and only if ${\rm det}(\lambda -x)=0.$ It is easy to deduce that  
 the class  of linear determinant-preserving bijections  coincides with class of linear spectrum-preserving
bijections.

However, for type II$_1$-factors $\cM$,  the situation changes dramatically. That is, 
there exists an element $x$ in a  type II$_1$-factor such that
${\rm det}(x-\lambda\mathbf{1})>0$ for all complex scalar $\lambda.$

Let $\{e_t\}_{0\le t\le 1}$ be a family  of projections in $\cM$ satisfying
\begin{align*}
e_t\le e_s,\, t\le s,\,\, \tau(e_t)=t,\, 0\le t \le 1.
\end{align*}

Define $x:=\int\limits_0^1 tde_t.$ Since $\tau(e_t)=t$ for all $0\le t\le 1,$ it follows that  the Brown measure
$\mu_x$ coincides with the Lebesgue measure on the unit interval $[0,1]$ (see \eqref{normalBrown}).
Note that the spectrum of $x$ is $\sigma(x)=[0,1], $ while  ${\rm det}(x-\lambda \mathbf{1})>0$ for all $\lambda \in \mathbb{C}.$
Indeed,  by   \eqref{Ldet}, for $\lambda \in [0,1]$,  we have 
\begin{align*}
{\rm det}(\lambda\mathbf{1} -x)=\exp\left(\int\limits_0^1 {\rm log}|\lambda - t|dt\right)=
\begin{cases}
e^{-1}, & \text{if  }\lambda=0, 1\\
e^{-1}\lambda^\lambda (1-\lambda)^{1-\lambda}, & \text{if   } 0<\lambda<1, 
 \end{cases}
\end{align*}
and,   for $\lambda \notin [0,1]$,  we have
\begin{align*}
{\rm det}(\lambda\mathbf{1} -x) & =\exp\left(\int\limits_0^1 {\rm log}|\lambda - t|dt\right)\ge c,
\end{align*}
where $c=\inf\limits_{0\le t \le 1}|\lambda - t|>0.$

\end{example}

Recall that the singular value function $\mu(x)$ of $x\in S(\cM)$ can be expressed as follows \cite[Section 4.1]{BCLSZ}:
\begin{align*}
\mu(t; x) = \inf\left\{s\ge 0: \tau_\cM(e_{(s, \infty)})\le t\right\}.
\end{align*}
The space $S(\cM)$ is a complete, metrizable, topological $*$-algebra. 
Indeed, it can be
equipped with a symmetric $F$-norm, by which the topology induced
is equivalent to the measure topology \cite[Proposition 4.2.5]{BCLSZ} (see also \cite{HS20}):
\begin{align*}
\left\|x\right\|_{L_0}=\inf\limits_{t>0}\left\{t+\mu(t;x)\right\}=\inf\limits_{\lambda>0}\left\{\lambda+\tau_\cM(e_{(\lambda,\infty)}(|x|))\right\},\, x\in S(\cM).
\end{align*}

Combining Theorem~\ref{HKcon}, Corollary~\ref{isometryvNa}, \cite[Corollary~1.5]{AK2020} and \cite[Corollary 4.4]{BHS24}, we obtain the next isomorphic characterization of von Neumann factors of type II$_1.$

\begin{corollary} Let $\cM$ and $\cN$ be von Neumann II$_1$-factors   with faithful normal tracial states $\tau_\cM$ and $\tau_\cN,$
respectively.
Then the following assertions are equivalent.
\begin{enumerate}
\item The projection lattices
$P(\cM)$ and $P(\cN)$ are lattice isomorphic;
\item $\cM$ and $\cN$ are  Jordan $*$-isomorphic;
\item $\cM$ and $\cN$ are $L_0$-isometric;
\item $\cM$ and $\cN$ are rank metric isometric;
\item there is a unital linear determinant-preserving bijection from $\cM$ onto $\cN.$
\end{enumerate}
\end{corollary}

In \cite[Lemma 4.2]{BHS24}, a   connection between the measure metric and rank metric was established.
For any $x\in S(\cM)$ we have
\begin{align*}
\lim\limits_{\lambda \to \infty}\left\|\lambda x\right\|_{L_0} = \tau(r(x))=\left\|x\right\|_S.
\end{align*}
This shows that every measure metric isometry between Murray--von Neumann algebras
is also a rank metric isometry. However, the converse does not hold in general.
For example, if $a$ is not a scalar multiple of a unitary operator, the right multiplication operator $R_a$
acts as a rank metric isometry but not an $L_0$-isometry\cite[Corollary 4.4]{BHS24}.

The general form of surjective additive mappings on $\mathbb{M}_n$ 
that preserve rank-one matrices was characterized in~\cite[Corollary]{OS93}. Precisely, if $\phi: \mathbb{M}_n\to \mathbb{M}_n$ is a surjective additive mapping preserving rank-one matrices, there exists a ring automorphism
$h : \bC \to  \bC$ such that $\phi$ is either of the form
\begin{center}
    $\phi\left([\lambda_{ij}]\right)=a\left[h\left(\lambda_{ij}\right)\right]b$ or
$\phi\left([\lambda_{ij}]\right)=a\left[h\left(\lambda_{ij}\right)\right]^t b$
\end{center}
for all $[\lambda_{ij}]\in \mathbb{M}_n,$ where $a$
and $b$ are invertible matrices, $x^t$ denotes the transpose of~$x.$

Notably,  $\mathbb{M}_n$ admits nonlinear rank-preserving mappings because the complex field $\mathbb{C}$ permits nonlinear ring automorphisms (see \cite[Example 3.4]{OS93}).

Furthermore, Bresar’s characterization of homomorphisms of matrix algebras~\cite[Corollary 3.2]{Bresar07},
applies to matrix algebras of order $\ge 2.$ Consequently, the proofs of Theorems~\ref{HKcon} and~\ref{isometry} remain valid for type I$_n$-factors with $n\ge 2.$ Thus,
a combination of 
 Theorem~\ref{HKcon} and Theorem~\ref{isometry} provides an operator-algebraic  viewpoint on  Frobenius' result, which 
characterizes
 bijective  linear  determinant-preserving  mappings and  rank-preserving  linear mappings, respectively,  on matrix algebras of order $n\ge 2$.

{\bf Acknowledgments.}
The first and the second authors were supported the NNSF of China (No.12031004, 12301160 and 12471134).
The third author was supported by the ARC.  The authors would like to thank Professor Aleksey Ber for useful discussions and helpful comments on earlier versions of this text.

\end{document}